\documentclass[11pt]{article}
\usepackage{amsmath,amsthm,amssymb}
\usepackage{latexsym}
\usepackage{mathrsfs}

\usepackage{enumerate}
\usepackage[shortlabels]{enumitem}
\usepackage{verbatim}

\usepackage[english]{babel}
\usepackage[ansinew]{inputenc}

\usepackage{graphics,graphicx}
\usepackage{subfig} 
\usepackage{subfloat}
\usepackage{booktabs}
\usepackage{multirow}
\usepackage{threeparttable}
\allowdisplaybreaks[3] 
\usepackage{float}

\usepackage{algorithm}
\usepackage{algorithmic}

\usepackage{bbold}

\usepackage{appendix}

\usepackage{authblk}

\usepackage{color}
\usepackage{bm}
\usepackage{url}
\usepackage{hyperref}
\hypersetup{
  colorlinks=true,
  linkcolor=blue,
  filecolor=blue,
  anchorcolor=blue,
  urlcolor=cyan,
  citecolor=purple
}

\usepackage{tikz}
\usepackage{pgfplots}
\usepackage{pifont}

\newtheorem{theorem}{Theorem}[section]
\newtheorem{cor}[theorem]{Corollary}
\newtheorem{lemma}[theorem]{Lemma}
\newtheorem{prop}[theorem]{Proposition}

\theoremstyle{definition}
\newtheorem{remark}{Remark}[section]
\newtheorem*{example}{Example}

\numberwithin{equation}{section}

\newcommand{\od}{\mathcal{O}(d)}
\newcommand{\odn}{\mathcal{O}(d)^{n}}
\newcommand{\sod}{\mathcal{SO}(d)}
\newcommand{\sodn}{\mathcal{SO}(d)^{n}}
\newcommand{\og}{\mathcal{G}}
\newcommand{\ogn}{\mathcal{G}^{n}}

\newcommand{\grad}{\operatorname{grad}}
\newcommand{\Hess}{\operatorname{Hess}}

\newcommand{\Exp}{\operatorname{Exp}}

\newcommand{\bQ}{\bm{Q}}
\newcommand{\bG}{\bm{G}}

\newcommand{\bC}{\bm{C}}

\begin{document}
\title{\bf\Large Rotation Group Synchronization via Quotient Manifold}
\author[1]{Linglingzhi Zhu\thanks{llzzhu@se.cuhk.edu.hk}}
\author[2]{Chong Li\thanks{cli@zju.edu.cn}}
\author[1]{Anthony Man-Cho So\thanks{manchoso@se.cuhk.edu.hk}}
\affil[1]{The Chinese University of Hong Kong}
\affil[2]{Zhejiang University}
\date{\today}

\maketitle

\begin{abstract}
Rotation group $\sod$ synchronization is an important inverse problem and has attracted intense attention from numerous application fields such as graph realization, computer vision, and robotics. In this paper, we focus on the least-squares estimator of rotation group synchronization with general additive noise models, which is a nonconvex optimization problem with manifold constraints. 
Unlike the phase/orthogonal group synchronization, there are limited provable approaches for solving rotation group synchronization. First, we derive improved estimation results of the least-squares/spectral estimator, illustrating the tightness and validating the existing relaxation methods of solving rotation group synchronization through the optimum of relaxed orthogonal group version under near-optimal noise level for exact recovery. 
Moreover, departing from the standard approach of utilizing the geometry of the ambient Euclidean space,
we adopt an intrinsic Riemannian approach to study orthogonal/rotation group synchronization. Benefiting from a quotient geometric view, we prove the positive definite condition of quotient Riemannian Hessian around the optimum of orthogonal group synchronization problem, and consequently the Riemannian local error bound property is established to analyze the convergence rate properties of various Riemannian algorithms.  
As a simple and feasible method, the sequential convergence guarantee of the (quotient) Riemannian gradient method for solving orthogonal/rotation group synchronization problem is studied, and we derive its global linear convergence rate to the optimum with the spectral initialization. 
All results are deterministic without any probabilistic model.
\end{abstract}


\section{Introduction}


The synchronization problem refers to estimating a collection of elements by pairwise mutual measurements. For orthogonal (sub)group $\og\subseteq\od$ synchronization, the task is to derive the target group element $\bG^\star:=[\bG_1^\star;\ldots;\bG_n^\star]\in\ogn$ that we collectively refer to as the ground truth from the noisy observations:
\[
\bC= \bG^\star \bG^{\star\top}+\bm{\Delta},
\]
where $\bm{\Delta}\in\operatorname{Sym}(nd)$ is an symmetric perturbation matrix. 
In this paper, we focus on the synchronization problem on the rotation group (i.e., $\og=\sod$), which is an important case of orthogonal subgroup synchronization. 
The rotation group synchronization is intensively considered in various areas including sensor network localization \cite{cucuringu2012sensor}, structural biology \cite{cucuringu2012eigenvector}, computer vision (e.g., point set registration \cite{khoo2016non,bohorquez2020maximizing}, multiview structure from motion \cite{arie2012global}, rotation averaging \cite{hartley2013rotation,dellaert2020shonan}), cryo-electron microscopy \cite{singer2011three,singer2018mathematics}, and also simultaneous localization and mapping for robotics \cite{rosen2019se}.



It is a common approach to study the orthogonal (sub)group synchronization via the least-squares estimator (also as the maximum likelihood estimator given Gaussian noise setting as a statistical model), which is the following optimization problem: 
\begin{equation*}
\begin{array}{ll}
\min\limits_{\bG\in\mathbb{R}^{nd\times d}} &\sum\limits_{i<j}\Vert \bG_i \bG_j^\top -\bC_{ij} \Vert_F^2\\
\ \ \ {\rm s.t.} &\bG_i\in \og,\ i\in[n].
\end{array}
\end{equation*}
By the orthogonality of $\bG_i$ for $i\in[n]$, the above optimization problem can be reformulated as
\begin{equation}\label{Problem00}\tag{Sync}
\max\limits_{\bG\in\ogn} 
\operatorname{tr}(\bG^\top \bC \bG).
\end{equation}
When it specialized to the phase synchronization $\og=\mathcal{SO}(2)$ (as $\mathcal{SO}(2)$ is isomorphic to $\mathcal{U}(1)$) \cite{singer2011angular,boumal2016nonconvex,bandeira2017tightness,liu2017estimation,zhong2018near} and orthogonal group synchronization $\og=\od$ \cite{liu2020unified,won2022orthogonal,ling2022solving,ling2022improved,ling2022near,zhu2021orthogonal}, the estimation performance and convergence guarantee of various algorithms for solving problem \eqref{Problem00} has already been extensively investigated in the literature. Also, there are results considering the cyclic group $\mathcal{Z}_m$ and permutation group $\mathcal{P}(d)$ \cite{liu2020unified,ling2022near}. We briefly review the existing approaches in literature for solving the synchronization problem as follows.


The most common approaches for solving \eqref{Problem00} is by relaxation including: spectral relaxation, semidefinite relaxation (i.e., convex relaxation) and Burer-Monteiro factorization as a certain type of nonconvex relaxation. The spectral relaxation \cite{singer2011angular,romanov2020noise,ling2022near} is a convenient method with guarantees for orthogonal group synchronization by simply computing the top $d$ eigenvectors closely approximating the ground truth, which is proved in \cite{zhang2022exact} as an estimator achieving minimax lower bound with Gaussian noise for phase and orthogonal group synchronization. Hence, it is usually used to be the initialization of other methods for deriving the least-squares estimator.
On the other hand, as a quadratic program with quadratic constraints (QPQC), the semidefinite programming relaxation (SDR) technique \cite{luo2010semidefinite,singer2011angular} of problem \eqref{Problem00} is a natural idea for possibly employed. With the help of the generative model, the SDR has been proven to be tight if the noise strength is relatively small for phase synchronization \cite{bandeira2017tightness} and orthogonal group $\od$ synchronization \cite{ling2022solving,won2022orthogonal}, and the noise level is improved to be near-optimal by the leave-one-out technique under Gaussian noise \cite{zhong2018near,ling2022improved}. Also, with the measurements corrupted by orthogonal matrices, the near-optimal performance guarantee for rotation group $\sod$ synchronization is derived \cite{wang2013exact}. However, a major drawback of the SDR is that it fails to scale well and is computationally expensive. Instead, the Burer-Monteiro factorization \cite{burer2003nonlinear,burer2005local}, as a nonconvex relaxation approach related to low-rank matrix optimization, offers better scalability compared to the SDR and has been proved that its first and second-order necessary optimality conditions are sufficient to guarantee global optimality \cite{ling2022solving} provided that the noise is sufficiently small for orthogonal group $\od$ synchronization. Taking advantage of the low-rank solution, it is more practical to use fast low-rank nonconvex optimization approach for solving the synchronization problem including Riemannian optimization \cite{dellaert2020shonan,mei2017solving,rosen2019se}.

In addition to various relaxation approaches, the generalized power method (GPM) \cite{journee2010generalized} as a projected gradient method to the product manifold $\ogn$ in the ambient Euclidean space resolves the nonconvex problem \eqref{Problem00} 
directly. 
For the phase synchronization problem, \cite{boumal2016nonconvex} proves that 
the GPM with spectral initialization converges to a global optimum of problem \eqref{Problem00}. Moreover, the linear convergence rate and improved noise bounds are derived in \cite{liu2017estimation,zhong2018near}. Later on, similar results are extended to orthogonal group $\od$ synchronization \cite{zhu2021orthogonal,ling2022improved}. 
Recently, there are also alternative models with algorithms utilizing the message-passing procedure \cite{perry2018message,lerman2022robust} to solve group synchronization. The summary and comprehensive comparisons of above-mentioned works have been discussed in the literature, 
and we refer the reader to \cite{liu2020unified,ling2022improved,zhu2021orthogonal} for details. 

From the above existing approaches it can be observed that
there are seldom results for rotation group $\sod$ synchronization besides the estimation performance about the GPM studied in \cite{liu2020unified}, iterative polar decomposition algorithm in \cite{gao2023optimal} achieving the minimax optimal under Gaussian noise, and the convergence of Riemannian gradient method for the phase synchronization problem \cite{chen2019first} which highly depends on the commutative nature of $\mathcal{SO}(2)$. 


In this paper, we study the rotation group synchronization problem
\begin{equation}\label{ProblemR}\tag{Sync-R}
\max\limits_{\bG\in\sodn} 
\operatorname{tr}(\bG^\top \bC \bG).
\end{equation}
Rather than directly solving problem \eqref{ProblemR}, we first focus on the properties of the relaxed orthogonal group $\od$ synchronization problem
\begin{equation}\label{Problem0}\tag{Sync-O}
\max\limits_{\bG\in\odn} \bar{f}(\bG):=\operatorname{tr}(\bG^\top \bC \bG),\quad \text{where}\ \bar{f}: \odn \rightarrow \mathbb{R}.\\
\end{equation}
Under near-optimal noise level of dependence on $n$ for exact
recovery, we prove the relaxation is tight, i.e., for the given ground truth $\bG^\star\in\sodn$, all the optimal solutions $\widehat{\bG}$ of problem \eqref{Problem0} satisfy $\widehat{\bG}\in\sodn$.
Then, we investigate the manifold constrained optimization problem \eqref{Problem0} from the intrinsic Riemannian perspective, which turns a constrained problem to a unconstrained one on the manifold and enjoys the advantages of Riemannian approach for dimension reduction. 
The additional cost about geometry information e.g., tangent space
and geodesic calculus with the exponential map can be compensated
thanks to the properties of given product orthogonal matrix 
$\odn$
based on Lie theory. 
On the other hand, the Riemannian algorithms will automatically keep the iterates on a single connected component of the orthogonal group and are naturally fit for the case of the rotation group $\sod$ synchronization by always providing a feasible solution. 

The central idea of this paper is to utilize the natural quotient structure of the problem \eqref{Problem0}. 
Since the function value $\bar{f}$ at $\bG\in \odn$ is invariant through any orthogonal group operation $\bQ\in\od$, it leads to the following equivalent optimization problem on the quotient manifold $\mathcal{Q}:=\odn/\od$  as follows:
\begin{equation}\label{Problem}\tag{Sync-Q}
  \max\limits_{[\bG]\in\mathcal{Q}}\ f([\bG]):=\operatorname{tr}(\bG^\top \bC \bG),\quad \text{where}\ f: \mathcal{Q}\rightarrow \mathbb{R},
  \end{equation}
and $[\bG]:=\left\{\bG' \in \odn \mid \bG'=\bG \bQ,\ \bQ \in \od\right\}$ is the equivalent class containing $\bG$. 
The main contribution is summerized as follows:
\begin{itemize}
\item The quotient geometry of the manifold $\mathcal{Q}$ is investigated including the horizontal space, quotient Riemannian metric, geodesic with exponential map on quotient manifold, and also the characterizations for optimality conditions with explicit forms of quotient Riemannian gradient and Hessian are provided, which is a crucial step for landscape analysis and algorithms design of the synchronization problem. The local positive definite condition of quotient Riemannian Hessian of problem \eqref{Problem} around the optimal solution 
are given from the intrinsic quotient Riemannian perspective, which implies the Riemannian local error bound property of the original problem \eqref{Problem0}. Note that this result is different from the one derived in \cite{zhu2021orthogonal} related to the ambient Euclidean space, and it can be applied to analyze the convergence behavior of various Riemannian algorithms.
\item Improved estimation results of the least-squares and spectral estimator are derived under the general deterministic additive noise model. As a consequence, 
it is tight to solve the relaxed problem \eqref{Problem0} for problem \eqref{ProblemR} under near-optimal noise level of dependence on $n$ for exact recovery, which makes all previous about 
the spectral method \cite{ling2022near}, the SDR \cite{won2022orthogonal,ling2022solving}, the Burer-Monteiro factorization \cite{boumal2016non,ling2022solving} and the GPM \cite{liu2020unified,zhu2021orthogonal,ling2022improved} for orthogonal group $\od$ synchronization can be applied to the rotation group $\sod$ synchronization. Also, the spectral estimator can also be proved to locate in the effective domain of Riemannian local error bound, which fits the initialization requirement of different algorithms' global convergence rate results. 
\item The convergence property of (quotient) Riemannian gradient method is investigated for solving problem \eqref{Problem0} (also \eqref{ProblemR}). By identifying the iterates keeping in the effective region of the Riemannian local error bound property for such a geodesically nonconvex problem, we show that it converges linearly to the global optimum with spectral initialization. Importantly, this convergence result is not restricted to the convergence in the sense of the equivalent class but also valid for the sequence of iterates.
\end{itemize}
In a word, this work provides an example to analyze algorithms from the intrinsic Riemannian view and utilize the quotient space of the equivalence class to simplify the landscape analysis of optimization problems.



\subsection{Notation}
Throughout the paper, we use the standard notations. Let the Euclidean space of all $m\times n$ real matrices $\mathbb{R}^{m\times n}$ be equipped with inner product $\langle \bm{X},\bm{Y}\rangle:=\operatorname{tr}(\bm{X}^\top \bm{Y})$ for any $\bm{X},\bm{Y}\in\mathbb{R}^{m\times n}$ and its induced Frobenius norm $\Vert\cdot\Vert_F$.
Let $\operatorname{Sym}(n)$ and $\operatorname{Skew}(n)$ be the space of $n\times n$ symmetric and skew-symmetric matrices, respectively. Let $\Vert \cdot\Vert_{*}$ and $\Vert \cdot\Vert$ be the nuclear norm and operator norm, respectively.
Given a real matrix $\bm{X} \in \mathbb{R}^{nd \times nd}$ with $n$ rows of $n$ $d \times d$ blocks each, we use $\bm{X}_{ij}$ (where $i, j \in [n]$) to denote its $(i,j)$-th $d\times d$ block and $\bm{X}_{i,:}$ (where $i \in [n]$) to denote its $i$-th $d \times nd$ block row. Given a real matrix $\bm{X} \in \mathbb{R}^{nd \times d}$ with $n$ $d \times d$ blocks stacked in a column, we use $\bm{X}_i$ (where $i \in [n]$) to denote the $i$-th $d \times d$ block and set $\|\bm{X}\|_{\infty} :=  \max_{i\in[n]} \left\| \bm{X}_i \right\|_F$.



\subsection{Organization}
The remainder of this paper is organized as follows. In section \ref{section-pre}, basic facts on Riemannian manifolds and some preliminary results for the orthogonal/rotation group are presented. Properties of orthogonal group synchronization \eqref{Problem0} from the quotient geometric view are discussed in section \ref{section-quo-geometry}. The improved estimation performance which plays a important role for validating existing relaxation approaches for solving rotation group synchronization is shown in section \ref{section-improvedest}. The landscape analysis of problem \eqref{Problem} from the quotient manifold is provided in section  \ref{section-landscape}. In section \ref{section-rie-algos}, we show the convergence result of the (quotient) Riemannian gradient method with spectral initialization
under a general additive noise model. 
We end with conclusions and future directions in section \ref{section-con}.

\section{Preliminaries}
\label{section-pre}

\subsection{Basic Facts on Riemannian Manifolds}
First, we recall some basic concepts and results on Riemannian manifolds. Let $\mathcal{M}$ be a complete connected $d$-dimensional Riemannian manifold. 
The tangent space at $\bm{X}\in\mathcal{M}$ is denoted by $\operatorname{T}_{\bm{X}} \mathcal{M}$ 
and the tangent bundle of $\mathcal{M}$ is denoted by $\operatorname{T} \mathcal{M}:=\bigcup_{\bm{X} \in \mathcal{M}} \operatorname{T}_{\bm{X}} \mathcal{M}$. 
Let $\langle\cdot,\cdot\rangle_{\bm{X}}$ be a Riemannian metric on $\mathcal{M}$  with 
the induced norm on $\operatorname{T}_{\bm{X}} \mathcal{M}$ for each  $\bm{X} \in \mathcal{M}$ that $\left\|\bm{\xi}_{\bm{X}}\right\|_{\bm{X}}:=\sqrt{\langle\bm{\xi}_{\bm{X}}, \bm{\xi}_{\bm{X}}\rangle_{\bm{X}}}$. Denote the Riemannian (i.e., Levi-Civita) connection $\nabla$ for the Riemannian manifold $\mathcal{M}$. For any two  points $\bm{X},\bm{Y}\in \mathcal{M}$, let $\gamma:[0,1]\rightarrow \mathcal{M}$ be a smooth curve connecting $\bm{X}$ and $\bm{Y}$. Then the arc-length of $\gamma$ is defined by $l(\gamma):=\int_{0}^{1}\Vert \dot{\gamma}(t)\Vert\mathrm{d}t$, and the Riemannian distance from $\bm{X}$ to $\bm{Y}$ by $\operatorname{d}^{\mathcal{M}}(\bm{X},\bm{Y}):=\inf_{\gamma}l(\gamma)$, where the infimum is taken over all piecewise smooth curves $\gamma:[0,1]\rightarrow \mathcal{M}$ connecting $\bm{X}$ and $\bm{Y}$. 
For a smooth curve $\gamma$, 
if $\dot{\gamma}$ is parallel along itself (i.e., $\nabla_{\dot{\gamma}}{\dot{\gamma}}=0$), then  $\gamma$ is called a geodesic.
A geodesic $\gamma:[0,1]\rightarrow \mathcal{M}$ joining $\bm{X}$ to $\bm{Y}$ is minimal if its arc-length equals its Riemannian distance between $\bm{X}$ and $\bm{Y}$. 
Also, up to parameterization, all minimizing curves 
are geodesics \cite[Theorem 6.4]{lee2018introduction}.
 By the Hopf-Rinow theorem,  there is at least one minimal geodesic joining $\bm{X}$ to $\bm{Y}$ for any points $\bm{X}$ and $\bm{Y}$ for a complete metric space $(\mathcal{M},\operatorname{d}^{\mathcal{M}})$.
The exponential map Exp: $\operatorname{T} \mathcal{M} \rightarrow \mathcal{M}$ is defined by
\[\operatorname{Exp}(\bm{X}, \bm{\xi_X})=\operatorname{Exp}_{\bm{X}}(\bm{\xi_X})=\gamma_{\bm{\xi_X}}(1),\]
where $\gamma_{\bm{\xi_X}}: [0,1] \rightarrow \mathcal{M}$ is a geodesic satisfying $\gamma_{\bm{\xi_X}}(0)=\bm{X}$ and $\dot{\gamma}_{\bm{\xi_X}}(0)=\bm{\xi_X}$, and $\operatorname{Exp}_{\bm{X}}$ is the restriction defined on $\operatorname{T}_{\bm{X}} \mathcal{M}$. For each $\bm{X}\in \mathcal{M}$, the exponential map at $\bm{X}$, $\Exp_{\bm{X}}:\operatorname{T}_{\bm{X}} \mathcal{M}\rightarrow \mathcal{M}$  is well-defined and smooth on $\operatorname{T}_{\bm{X}}\mathcal{M}$ \cite[Proposition 5.19]{lee2018introduction}. The concept of the geodesically (strongly) convex set is consistent with \cite[Definition 11.2 and 11.17]{boumal2023introduction}.

Let $f: \mathcal{M} \rightarrow \mathbb{R}$ be a smooth function. Then the Riemannian gradient of $f$ is a vector field $\operatorname{grad} f$ as  the unique element  in $\operatorname{T}_{\bm{X}} \mathcal{M}$ for any $\bm{X} \in \mathcal{M}$ such that
\[
\langle\grad f(\bm{X}), \bm{\xi_{X}}\rangle_{\bm{X}}=\operatorname{D} f(\bm{X})\left[\bm{\xi_{X}}\right]\quad \text{for each}\ \bm{\xi_{X}} \in \operatorname{T}_{\bm{X}} \mathcal{M},
\]
where $\operatorname{D}f$ is the differential of the function $f$. The Riemannian Hessian of $f$ at $\bm{X} \in \mathcal{M}$ is defined as the linear mapping from $\operatorname{T}_{\bm{X}} \mathcal{M}$ to $\operatorname{T}_{\bm{X}} \mathcal{M}$ such that 
\[
\Hess f(\bm{X})\left[\bm{\xi_{X}}\right]=\nabla_{\bm{\xi_{X}}} \grad f(\bm{X})\quad \text{for each}\ \bm{\xi_{X}} \in \operatorname{T}_{\bm{X}} \mathcal{M}.
\]
The concept of the geodesically (strongly) convex function is consistent with \cite[Definition 11.3 and 11.5]{boumal2023introduction}.

\subsection{Properties of Orthogonal/Rotation Group}
Now, we introduce some basic results about orthogonal/rotation group. 
The orthogonal/rotation group $\og$ is an embedded submanifold of $\mathbb{R}^{d\times d}$ \cite[Section 3.3.2]{absil2009optimization}. 
We consider the Riemannian metric on $\og$ that is induced from the Euclidean inner product, i.e, at $\bm{U}\in \og$ we have $\langle\bm{\xi}, \bm{\eta}\rangle_{\bm{U}}:=\operatorname{tr}(\bm{\xi}^{\top} \bm{\eta})$ for any $\bm{\xi}, \bm{\eta} \in \operatorname{T}_{\bG} \og$. For the product Riemannian manifold $\ogn$, the Riemannian metric is defined by $\langle \bm{\xi},\bm{\eta}\rangle_{\bG}=\sum_{i=1}^n \operatorname{tr}(\bm{\xi}_i^{\top} \bm{\eta}_i)$ for any $\bm{\xi}:=[\bm{\xi}_1,\ldots,\bm{\xi}_n]$, $\bm{\eta}:=[\bm{\eta}_1,\ldots,\bm{\eta}_n] \in \operatorname{T}_{G} \ogn$ at $\bG\in\ogn$.
From the definition we know that
\begin{align}
\mathcal{O}(d)^n&=\{\bG\in\mathbb{R}^{nd\times d}\mid \bG=[\bG_1;\ldots;\bG_n],\ \bG_i \bG_i^\top=\bG_i^\top \bG_i =\bm{I}_d, \ i\in [n]\}\notag
\end{align}
and  $\mathcal{SO}(d)^n=\{\bG\in\mathcal{O}(d)^n\mid  \operatorname{det}(\bG_i)=1, \ i\in [n]\}$.
 Note that the tangent space to $\ogn$ at $\bG \in \ogn$ is given by
\begin{align}
  \operatorname{T}_{\bG}\ogn 
  &=\{\bm{H}=\left[\bm{H}_{1} ; \ldots ; \bm{H}_{n}\right] \in \mathbb{R}^{n d \times d}  \mid  \bm{H}_{i}=\bm{G}_{i}\bm{E}_{i} , \bm{E}_{i}=-\bm{E}_{i}^{\top}, i \in[n]\}.\notag
  \end{align}
The projection 
onto the orthogonal/rotation group $\og$, denoted as $\Pi_{\og}(\cdot)$, which in particular has a closed-form solution 
(cf.\ \cite{liu2020unified}) that for each $\bm{Z}\in \mathbb{R}^{d\times d}$ with singular value decomposition $\bm{Z} = \bm{U}_{\bm{Z}}\bm{\Sigma}_{\bm{Z}}\bm{V}_{\bm{Z}}^\top$, 
\begin{align} 
   \Pi_{\mathcal{O}(d)}(\bm{Z})&:=\left\{\bm{X}\in \mathcal{O}(d) \ \bigg| \ \Vert \bm{X}-\bm{Z}\Vert_F=\inf_{\bm{Y}\in \od}\|\bm{Y}-\bm{Z}\|_F\right\} = \bm{U}_{\bm{Z}} \bm{V}_{\bm{Z}}^\top,  \label{eq: orthoprocru}\\
   \Pi_{\sod}(\bm{Z})&:=\left\{\bm{X}\in \sod \ \bigg| \ \Vert \bm{X}-\bm{Z}\Vert_F=\inf_{\bm{Y}\in \sod}\|\bm{Y}-\bm{Z}\|_F\right\}=\bm{U}_{\bm{Z}} \bm{I}_{\bm{Z}} \bm{V}_{\bm{Z}}^{\top},\label{eq: sorthoprocru}
\end{align}  
where $\bm{I}_{\bm{Z}}:=\operatorname{Diag} ( [1;\ldots;1;\operatorname{det}(\bm{U}_{\bm{Z}} \bm{V}_{\bm{Z}}^{\top}) ] )\in \mathbb{R}^{d \times d}$.
Note that the solution to 
\eqref{eq: sorthoprocru} can be not unique as the polar decomposition of $\bm{Z}$ can be not unique.
Moreover, we define $ \Pi_{\ogn}(\bG) := [\Pi_{\og}(\bG_1);\ldots;\Pi_{\og}(\bG_n)]$, where $\bG=[\bG_1;\ldots;\bG_n]\in\mathbb{R}^{nd\times d}$.


The exponential map on $\og$ is $\Exp:\operatorname{T} \og\rightarrow \og$ given by  \cite[Example 4.12]{absil2009optimization} (also see \cite[(2.14)]{edelman1998geometry} for details)
\[
\Exp(\bm{U},\bm{UE})=\Exp_{\bm{U}}(\bm{UE})= \bm{U}\exp(\bm{E}) \quad \text{for any}\ (\bm{U},\bm{UE})\in \operatorname{T} \og.
\]
where  $\exp(\bm{E}):=\sum_{k=0}^{\infty} \frac{1}{k !} \bm{E}^{k}$ is the matrix exponential map, since  the Lie exponential map (i.e., matrix exponential) and the Riemannian exponential map coincide for bi-invariant Riemannian metric on the Lie group $\og$; see \cite[Problem 5-8]{lee2018introduction}.
Also, from the given Riemannian metric for the product manifold $\ogn$, we know that for any $\bG\in \ogn$ and $\bm{\xi}:=[\bG_1\bm{E}_1;\ldots;\bG_n \bm{E}_n]\in\operatorname{T} \ogn$, $\Exp:\operatorname{T} \ogn\rightarrow \ogn$ is defined by 
\[
\Exp(\bG,\bm{\xi})=\Exp_{\bG}(\bm{\xi})= [\bG_1\exp(\bm{E}_1);\ldots; \bG_n\exp(\bm{E}_n)].
\]
Then we introduce the following useful lemma characterizes the Lipschitz constant of the exponential map on the tangent space. We refer the reader to Appendix \ref{appendix-A} for proof details.
\begin{lemma}
\label{lemma-exp-lip}
Let $\bm{E}_1,\bm{E}_2\in\operatorname{Skew}(d)$. Then
\[
\|\exp(\bm{E}_2)-\exp(\bm{E}_1)\|_F\leq \|\bm{E}_2-\bm{E}_1\|_F.
\]
Furthermore, for any $\bG\in \ogn$ and $\bm{\xi}:=[\bG_1\bm{E}_1;\ldots;\bG_n \bm{E}_n] \in \operatorname{T}_{\bG} \ogn$, it follows that
\begin{align}
&\left\|\operatorname{Exp}_{\bG}(\bm{\xi})-\bG\right\|_{F} \leq \|\bm{\xi}\|_{F} 
,\label{retr-key-ineq-1}\\ 
&\left\|\operatorname{Exp}_{\bG}(\bm{\xi})-(\bG+\bm{\xi})\right\|_{F} \leq \frac{1}{2}\|\bm{\xi}\|_{F}^{2}
.\label{retr-key-ineq-2}
\end{align}
\end{lemma}





\section{Quotient Geometry of Orthogonal Group Synchronization}
\label{section-quo-geometry}
\subsection{Quotient Manifold $\mathcal{Q}$}
The quotient manifold $\mathcal{Q}:=\odn/\od$ (well-definedness via the quotient manifold theorem  
\cite[Theorem 21.10]{lee2012introduction}) is based on the following equivalence relation on $\odn$:
\[
\bG \sim \bG' \Longleftrightarrow\left\{\bG \bQ \mid \bQ \in \od \right\}=\left\{ \bG' \bQ \mid \bQ \in \od\right\}.
\]
Then we know that $\mathcal{Q}=\{[\bG] \mid \bG \in \odn\}$, where
$
[\bG]:=\left\{\bar{\bG} \in \odn \mid \bar{\bG}=\bG \bQ,\ \bQ \in \od\right\}.
$
Define the canonical projection $\pi: \odn \rightarrow \mathcal{Q}$ by
\[
\pi(\bG):=[\bG] \quad \text{for each}\ \bG \in \odn.
\]
Moreover, let $\pi^{-1}$ be the preimage of $\pi$, then for any $\bG\in\odn$ it follows that 
\[
[\bG]=\pi^{-1}(\pi(\bG))\quad \text{and}\quad \operatorname{dim} \pi^{-1}(\pi(\bG))=\operatorname{dim} \od =\frac{1}{2} d(d-1).
\]
Since $\odn$ is the total space of $\mathcal{Q}$, from \cite[Proposition 3.4.4]{absil2009optimization} one has that
\[
\begin{aligned}
\operatorname{dim} \mathcal{Q} &=\operatorname{dim} \odn-\operatorname{dim} \pi^{-1}(\pi(\bG)) 
=\frac{1}{2}(n-1)d(d-1) .
\end{aligned}
\]
Now, we are going to investigate the tangent space of the quotient manifold $\mathcal{Q}$. 
Although it is difficult to obtain the explicit formula of $\operatorname{T}_{[\bG]} \mathcal{Q}$, we can derive its lifted representation on the tangent space of $\odn$ as follows. For any $[\bG] \in \mathcal{Q}$, let $\bm{\xi}_{[\bG]}\in \operatorname{T}_{[\bG]} \mathcal{Q}$ and $\bar{\bG}\in\pi^{-1}([\bG])$.
Define 
\[
\bar{\bm{\xi}}_{\bar{\bG}} \in \operatorname{T}_{\bar{\bG}}\odn\quad \text{such that}\quad\operatorname{D} \pi(\bar{\bG})\left[\bar{\bm{\xi}}_{\bar{\bG}}\right]=\bm{\xi}_{[\bG]}
\]
as a representation of $\bm{\xi}_{[\bG]}\in \operatorname{T}_{[\bG]} \mathcal{Q}$ on $\operatorname{T}_{\bar{\bG}}\odn$. Since there are infinite representations $\bar{\bm{\xi}}_{\bar{\bG}}$ of $\bm{\xi}_{[\bG]}$ at $\bar{\bG}$, it is desirable to identify a unique lifted representation of tangent vectors of $\operatorname{T}_{[\bG]} \mathcal{Q}$ in $\operatorname{T}_{\bar{\bG}}\odn$. The equivalence class $\pi^{-1}([\bG])$ is an embedded submanifold of $\odn$ \cite[Proposition 3.4.4]{absil2009optimization}, then $\pi^{-1}([\bG])$ admits a tangent space called the vertical space at $\bar{\bG}$:
\[
\mathcal{V}_{\bar{\bG}}:=\operatorname{T}_{\bar{\bG}}(\pi^{-1}([\bG])).
\]
Let the subspace $\mathcal{H}_{\bar{\bG}}\subseteq \operatorname{T}_{\bar{\bG}}\odn$ such that $\mathcal{H}_{\bar{\bG}} \oplus \mathcal{V}_{\bar{\bG}}=\operatorname{T}_{\bar{\bG}} \odn$ be called the horizontal space at $\bar{\bG}$.
Once $\odn$ has a horizontal space at $\bar{\bG}$, there exists one and only one element $\bar{\bm{\xi}}_{\bar{\bG}}$ that belongs to $\mathcal{H}_{\bar{\bG}}$ and satisfies $\operatorname{D} \pi(\bar{\bG})\left[\bar{\bm{\xi}}_{\bar{\bG}}\right]=\bm{\xi}_{[\bG]}$, and the unique vector $\bar{\bm{\xi}}_{\bar{\bG}}\in \mathcal{H}_{\bar{\bG}}$ is called the horizontal lift of $\bm{\xi}_{[\bG]}\in\operatorname{T}_{[\bG]} \mathcal{Q}$ at $\bar{\bG}$.

\begin{figure}[ht]
    \centering
    \includegraphics[width=0.5\textwidth]{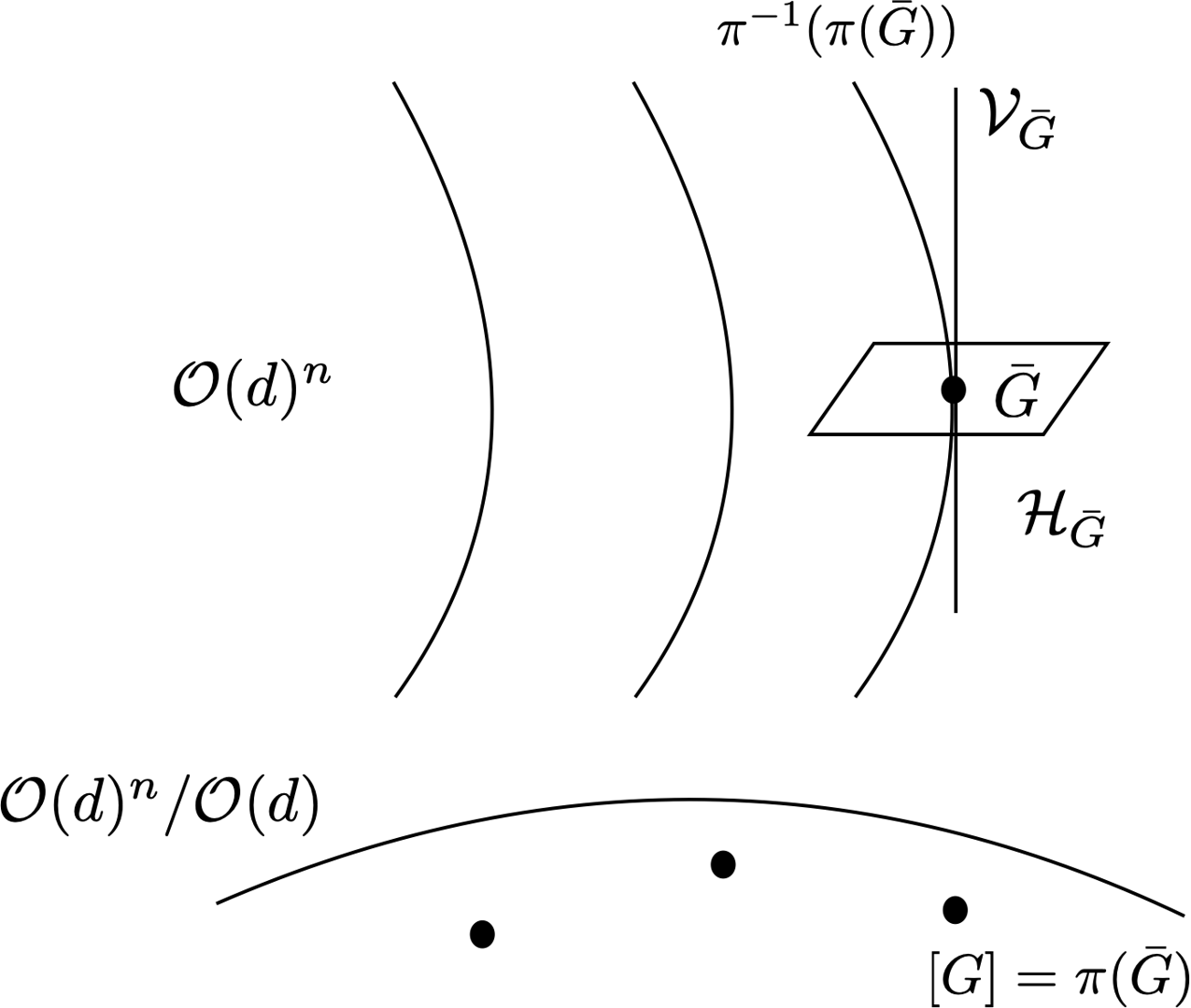}
    \caption{Tangent space of quotient manifolds}
    \label{fig:tangentspace}
\end{figure}

In order to derive the vertical space and horizontal space related to the quotient manifold $\mathcal{Q}$, we use the Riemannian metric $\langle\cdot,\cdot\rangle_{\bG}$ defined on $\odn$. Then we have the following calculation results.

\begin{lemma}
Let $[\bG]\in\mathcal{Q}$ and $\bar{\bG}\in \pi^{-1}([\bG])$. Then the vertical and horizontal space at $\bar{\bG}$ has the form that
\begin{align*}
\mathcal{V}_{\bar{\bG}}
&
=\left\{\bar{\bG} \bm{E} \ \Big| \ \bm{E}=-\bm{E}^\top,\ \bm{E} \in \mathbb{R}^{d \times d}\right\},\\
\mathcal{H}_{\bar{\bG}}
&=\left\{\bm{\xi}_{\bar{\bG}} \in \mathbb{R}^{n d \times d} \ \Bigg| \ (\bm{\xi}_{\bar{\bG}})_{i}=\bar{\bG}_{i} \bm{E}_{i},\ \bm{E}_{i}=-\bm{E}_{i}^{\top}, i \in[n]\ \text{and}\ \sum_{i=1}^n \bm{E}_i=\bm{0} \right\}.
\end{align*}
Moreover, the orthogonal projection of any element $\bm{\eta}_{\bar{\bG}} \in \operatorname{T}_{\bar{\bG}}\odn$ onto $\mathcal{H}_{\bar{\bG}}$ at $\bar{\bG}$ is given by
$$
\operatorname{Proj}_{\mathcal{H}_{\bar{\bG}}} (\bm{\eta}_{\bar{\bG}})=\left(\bm{I}_{nd}-\frac{1}{n} \bar{\bG}\bar{\bG}^\top  \right) \bm{\eta}_{\bar{\bG}}.
$$
\end{lemma}

Define a Riemannian metric $\langle\cdot, \cdot\rangle_{[\bG]}$ on the quotient manifold $\mathcal{Q}$ by
\[
\langle\bm{\xi}_{[\bG]}, \bm{\zeta}_{[\bG]}\rangle_{[\bG]}:=\langle \bar{\bm{\xi}}_{\bar{\bG}}, \bar{\bm{\zeta}}_{\bar{\bG}}\rangle=\operatorname{tr}(\bar{\bm{\xi}}_{\bar{\bG}}^\top \bar{\bm{\zeta}}_{\bar{\bG}}), \quad \text{for any}\ \bm{\xi}_{[\bG]}, \bm{\zeta}_{[\bG]} \in \operatorname{T}_{[\bG]} \mathcal{Q}, \ [\bG] \in \mathcal{Q},
\]
where $\bar{\bm{\xi}}_{\bar{\bG}}, \bar{\bm{\zeta}}_{\bar{\bG}} \in \mathcal{H}_{\bar{\bG}}$ are the unique horizontal lifts of $\bm{\xi}_{[\bG]}, \bm{\zeta}_{[\bG]}$ at $\bar{\bG}$, respectively. Since $\bar{\bG} \in \pi^{-1}([\bG])$, $\bar{\bG} \bQ\in \pi^{-1}([\bG])$ for any $\bQ \in \odn$, we can directly verify that
\[
\langle\bar{\bm{\xi}}_{\bar{\bG}\bQ}, \bar{\bm{\zeta}}_{\bar{\bG}\bQ}\rangle=\langle\bar{\bm{\xi}}_{\bar{\bG}}, \bar{\bm{\zeta}}_{\bar{\bG}}\rangle \quad \text{for each}\ \bQ \in \od.
\]



Next, we define the map $\Exp$ on $\mathcal{Q}$ (as an exponential map from \cite[Corollary 9.55]{boumal2023introduction}) as follows:
$$
\Exp_{[\bG]}(\bm{\xi}_{[\bG]}):=\pi (\overline{\Exp}_{\bar{\bG}}(\bar{\bm{\xi}}_{\bar{\bG}})) \quad \text{for each}\ \bm{\xi}_{[\bG]} \in \operatorname{T}_{[\bG]} \mathcal{Q},
$$
where $[\bG]=\pi(\bar{\bG}) \in \mathcal{Q}$, $\bar{\bm{\xi}}_{\bar{\bG}} \in \mathcal{H}_{\bar{\bG}}$ is the horizontal lift of a $\bm{\xi}_{[\bG]} \in \operatorname{T}_{[\bG]} \mathcal{Q}$ at $\bar{\bG}$, and $\overline{\Exp}$ is a exponential map on $\odn$.
Obviously, for the exponential map $\overline{\Exp}$ we have $\pi(\overline{\Exp}_{\bar{\bG}}(\bar{\bm{\xi}}_{\bar{\bG}}))=\pi(\overline{\Exp}_{\bar{\bG}'}(\bar{\bm{\xi}}_{\bar{\bG}'}))$ for all $\bar{\bG}, \bar{\bG}' \in \pi^{-1}([\bG])$. 




The functions $\bar{f}:\odn \rightarrow \mathbb{R}$ and $f: \mathcal{Q} \rightarrow \mathbb{R}$ defined in \eqref{Problem0} and \eqref{Problem} have the following relationship that
\[
f([\bG])=\bar{f}(\pi^{-1}([\bG])) \quad \text { and } \quad \bar{f}(\bG)=f(\pi(\bG)).
\]
Moreover,
\[
\operatorname{D} \bar{f}(\bar{\bG})\left[\bar{\bm{\xi}}_{\bar{\bG}}\right]=\operatorname{D} f(\pi(\bar{\bG}))\left[\operatorname{D} \pi(\bar{\bG})\left[\bar{\bm{\xi}}_{\bar{\bG}}\right]\right]=\operatorname{D} f([\bG])\left[\bm{\xi}_{[\bG]}\right].
\]
Finally, 
we denote the distance $\operatorname{d}_F(\cdot, \cdot)$ and $\operatorname{d}_{\infty}(\cdot, \cdot)$ based on quotient space. For each $[\bm{X}],[\bm{Y}]\in\mathbb{R}^{nd\times d}/\od$, they are defined as
\begin{align*}
	\operatorname{d}_{F}([\bm{X}],[\bm{Y}]) &:=  \min\limits_{\bQ\in \od} \Vert \bm{X}-\bm{Y}\bQ\Vert_F=\Vert \bm{X}-\bm{Y}\bQ^\star\Vert_F,\\
  \operatorname{d}_{\infty}([\bm{X}], [\bm{Y}])&:=
  \min\limits_{\bQ\in \od}\max_{i\in[n]} \Vert \bm{X}_i-\bm{Y}_i \bQ\Vert_F\leq 
  \Vert \bm{X}-\bm{Y}\bQ^\star\Vert_{\infty}=\max_{i\in[n]} \Vert \bm{X}_i-\bm{Y}_i \bQ^\star\Vert_F,
\end{align*}
where $\bQ^\star:=\operatorname{argmin}_{\bQ\in \od} \Vert \bm{X}-\bm{Y}\bQ\Vert_F$.
If $\bm{X},\bm{Y} \in \ogn$, then it follows that
\[
\operatorname{d}_F([\bm{X}],[\bm{Y}])^2 = 2\left(nd - \max_{\bQ\in \od } \langle \bm{Y}^\top \bm{X},\bQ \rangle\right) = 2(nd - \| \bm{Y}^\top \bm{X}\|_*).
\]

Now, we discuss the relationship among different notions of distance, i.e., the quotient distance $\operatorname{d}_F(\cdot, \cdot)$, $\operatorname{d}_{\infty}(\cdot, \cdot)$ and the Riemannian distance $\operatorname{d}^{\mathcal{G}}(\cdot,\cdot)$ and $\operatorname{d}^{\mathcal{Q}}(\cdot,\cdot)/\operatorname{d}^{\mathcal{Q}}_{\infty}(\cdot,\cdot)$ on the original manifold $\og$ and quotient manifold $\mathcal{Q}$, respectively.






\begin{lemma}\label{lemma-dist-com}
Let $\bm{X}, \bm{Y}\in \ogn$. 
Then it follows that $\operatorname{d}_{\infty}^{\mathcal{Q}}([\bm{X}],[\bm{Y}])\leq \operatorname{d}^{\mathcal{Q}}([\bm{X}],[\bm{Y}])$ and
\[
\operatorname{d}_{\infty}([\bm{X}], [\bm{Y}])\leq\operatorname{d}_F([\bm{X}],[\bm{Y}])\leq \operatorname{d}^{\mathcal{Q}}([\bm{X}],[\bm{Y}])\leq\operatorname{d}^{\mathcal{G}}(\bm{X},\bm{Y}).
\]
\end{lemma}
\begin{proof}
First, it can be directly verified from the definition that $\operatorname{d}_{\infty}^{\mathcal{Q}}([\bm{X}],[\bm{Y}])\leq \operatorname{d}^{\mathcal{Q}}([\bm{X}],[\bm{Y}])$ and $\operatorname{d}_{\infty}([\bm{X}], [\bm{Y}])\leq\operatorname{d}_F([\bm{X}],[\bm{Y}])$.
Let $\tilde{\bm{X}}\in\ogn$ be such that 
\[
\operatorname{d}^{\mathcal{G}}(\tilde{\bm{X}}, \bm{Y})=\inf_{\bm{X}^{\prime} \sim \bm{X}} \operatorname{d}^{\mathcal{G}}(\bm{X}^{\prime}, \bm{Y}).
\]
Then from \cite[Exercise 10.15]{boumal2023introduction} we know that $\operatorname{d}^{\mathcal{\bG}}(\tilde{\bm{X}}, \bm{Y})=\operatorname{d}^{\mathcal{Q}}([\bm{X}], [\bm{Y}])\leq\operatorname{d}^{\mathcal{G}}(\bm{X}, \bm{Y})$. On the other hand, it follows that 
\begin{align*}
\operatorname{d}_{F}([\bm{X}],[\bm{Y}]) 
&\leq \|\tilde{\bm{X}}-\bm{Y}\|_F
\leq\operatorname{d}^{\mathcal{G}}(\tilde{\bm{X}}, \bm{Y}),
\end{align*}
where the second inequality is from the fact that the intrinsic distance is larger than the extrinsic one for embedded Riemannian submanifold. 
The proof is complete.
\end{proof}

\subsection{Quotient Riemannian Gradient and Hessian}
In order to study the landscape and design algorithms for \eqref{Problem0} and \eqref{Problem}, we give explicit formulas of Riemannian gradient and Riemannian Hessian of the cost function $f$ in this section. Define the ancillary function $\tilde{f}: \mathbb{R}^{nd \times d} \rightarrow \mathbb{R}$ by
$$
\tilde{f}(\bm{X}):= \operatorname{tr}(\bm{X}^\top \bC \bm{X})\quad \text{for each}\ \bm{X} \in \mathbb{R}^{nd \times d}.
$$
Then $\bar{f}$ defined in \eqref{Problem0} is the restriction of $\tilde{f}$ onto $\odn$, i.e., $\bar{f}=\tilde{f}|_{\odn}$. Define a linear operator symblockdiag: $\mathbb{R}^{nd\times nd}\rightarrow \operatorname{Skew}(nd)$ by
\[ 
  \operatorname{symblockdiag}(\bm{X})_{ij} := \left\{
  \begin{array}{c@{\quad}l}
  \displaystyle \frac{\bm{X}_{ii}+\bm{X}_{ii}^\top}{2}, & \textrm{if } i=j, \\
  \noalign{\smallskip}
  \bm{0}, & \textrm{otherwise},
  \end{array}
  \right. \quad i, j \in [n].
\]
Let $\bm{X}\in\mathbb{R}^{nd\times d}$ and denote
\[
S(\bm{X}):=\operatorname{symblockdiag}(\bC \bm{X} \bm{X}^\top)-\bC\in\mathbb{R}^{nd\times nd}.
\]
Then we have the following calculation result about the (quotient) Riemannian gradient.

\begin{lemma}\label{grad-cal-lemma}
Let $[\bG] \in \mathcal{Q}$ and $\bar{\bG} \in \pi^{-1}([\bG])$. Then the unique horizontal lift of the Riemannian gradient $\grad f([\bG])$ of $f$ at $\bar{\bG} \in \odn $  is given by
\begin{align}\label{grad-cal}
\overline{\operatorname{grad} f([\bG])}_{\bar{\bG}}=\grad \bar{f}(\bar{\bG})=-2S(\bar{\bG})\bar{\bG}.
\end{align}
\end{lemma}

\begin{proof}
By simple calculation, the gradient of $\tilde{f}$ at $\bm{X} \in \mathbb{R}^{nd \times d}$ is given by
$\grad \tilde{f}(\bm{X})=2\bC\bm{X}$. Since $\odn$ is a Riemannian submanifold of $\mathbb{R}^{nd \times d}$ with Riemannian metric induced from the Euclidean inner product, one has the Riemannian gradient of $\bar{f}$ at $\bG\in \odn$ that
\[
\grad \bar{f}(\bG)=\operatorname{Proj}_{\operatorname{T}_{\bG}\odn}(\grad \tilde{f}(\bG))=2(\bC-\operatorname{symblockdiag}(\bC \bG \bG^\top )) \bG=-2S(\bG)\bG,
\]
where $\operatorname{Proj}_{\operatorname{T}_{\bG}\odn}$ indicates the orthogonal projection onto $\operatorname{T}_{\bG}\odn$, which is given by
\begin{equation*}
\operatorname{Proj}_{\operatorname{T}_{\bG}\odn}(\bm{X})=\bm{X}-\operatorname{symblockdiag}(\bm{X}\bG^\top)\bG, \quad \text{for each}\ \bm{X} \in \mathbb{R}^{nd \times d}.
\end{equation*}
Therefore we know that
\[
\overline{\operatorname{grad} f([\bG])}_{\bar{\bG}}=\operatorname{Proj}_{\mathcal{H}_{\bar{\bG}}}(\operatorname{grad} \bar{f}(\bar{\bG}))=\left(\bm{I}_{nd}-\frac{1}{n} \bar{\bG}\bar{\bG}^\top  \right) (-2S(\bar{\bG})\bar{\bG})=-2S(\bar{\bG})\bar{\bG}.
\]
The proof is complete.
\end{proof}

\begin{remark}\label{grad-equal}
From Lemma \ref{grad-cal-lemma} we observe that the Riemannian gradient method on the quotient manifold $\mathcal{Q}$ coincides with the one on the original manifold $\odn$, i.e.,
\begin{center}
$\overline{\operatorname{grad} f([\bG])}_{\bar{\bG}}=\grad \bar{f}(\bar{\bG})$ at any $\bar{\bG}\in \pi^{-1}([\bG])$. 
\end{center}
Since the function $\bar{f}$ is invariant on the equivalence classes $\bar{\bG} \in \pi^{-1}([\bG])$, we know that  
\begin{center}
$\operatorname{D} \bar{f}(\bar{\bG})\bar{\bm{\xi}}_{\bar{\bG}}=\langle \grad \bar{f}(\bar{\bG}),\bar{\bm{\xi}}_{\bar{\bG}}\rangle_{\bar{\bG}}=0$ for any $\bar{\bm{\xi}}_{\bar{\bG}}\in \mathcal{V}_{\bar{\bG}}=\operatorname{T}_{\bar{\bG}}(\pi^{-1}([\bG]))$. 
\end{center}
Therefore, $\operatorname{grad} \bar{f}(\bar{\bG}) \in\left(\mathcal{V}_{\bar{\bG}}\right)^{\perp}=\mathcal{H}_{\bar{\bG}}$ is exactly the horizontal lift of $\operatorname{grad} f([\bG])$ at $\bar{\bG}$. This helps us design Riemannian algorithms on the quotient manifold;  
see section \ref{section-rie-algos} for details.
\end{remark}

The following lemma shows the explicit form of quotient Riemannian Hessian vector.
\begin{lemma}
Let $[\bG] \in \mathcal{Q}$ and $\bar{\bG} \in \pi^{-1}([\bG])$. Then the unique horizontal lift of the Riemannian Hessian $\Hess f([\bG])$ of $f$ with direction $\bm{H}_{[\bG]}$ at $\bar{\bG} \in \odn $ is given by
\begin{align}\label{Hess-cal}
\overline{\operatorname{Hess} f([\bG])\left[\bm{H}_{[\bG]}\right]}_{\bar{\bG}}
&=\left(\bm{I}_{nd}-\frac{1}{n} \bar{\bG}\bar{\bG}^\top  \right)\left(\operatorname{Proj}_{\operatorname{T}_{\bar{\bG}}\odn }(-2S(\bar{\bG})\bar{\bm{H}}_{\bar{\bG}}) \right).
\end{align}
\end{lemma}
\begin{proof}
Recall that $\nabla$ and $\overline{\nabla}$ are the Riemannian connections on $\mathcal{Q}$ and $\odn$, respectively. The Riemannian Hessian of $f$ at $[\bG] \in \mathcal{Q}$ is given by
\[
\Hess f([\bG])\left[\bm{H}_{[\bG]}\right]=\nabla_{\bm{H}_{[\bG]}} \grad f([\bG]), \quad \text{for each}\ \bm{H}_{[\bG]} \in \operatorname{T}_{[\bG]} \mathcal{Q}.
\]
Thus, from \cite[(5.15) and Proposition 5.3.3]{absil2009optimization}, it follows that
\begin{equation}\label{Hess-def}
\begin{aligned}
\overline{\Hess f([\bG])\left[\bm{H}_{[\bG]}\right]}_{\bar{\bG}} &=\overline{\nabla_{\bm{H}_{[\bG]}} \grad f([\bG])}_{\bar{\bG}}=\operatorname{Proj}_{\mathcal{H}_{\bar{\bG}}}\left(\overline{\nabla}_{\bar{\bm{H}}_{\bar{\bG}}} \overline{\grad f([\bG])}_{\bar{\bG}}\right) \\
&=\operatorname{Proj}_{\mathcal{H}_{\bar{\bG}}}\left(\operatorname{Proj}_{\operatorname{T}_{\bar{\bG}}\odn}\left(\operatorname{D} \operatorname{grad} \bar{f}(\bar{\bG})\left[\bar{\bm{H}}_{\bar{\bG}}\right]\right)\right),
\end{aligned}
\end{equation}
where $\operatorname{D} \operatorname{grad} \bar{f}(\bar{\bG})\left[\bm{H}_{\bar{\bG}}\right]$ stands for the classical directional derivative. Since
\[
\begin{aligned}
  \operatorname{D}\operatorname{grad} \bar{f}(\bar{\bG})\left[\bar{\bm{H}}_{\bar{\bG}}\right]
&= \lim_{t\rightarrow 0}\frac{2(S(\bar{\bG})\bar{\bG}-S(\bar{\bG}+t\bar{\bm{H}}_{\bar{\bG}})(\bar{\bG}+t\bar{\bm{H}}_{\bar{\bG}}))}{t}\\
&=-2\left(S(\bar{\bG})\bar{\bm{H}}_{\bar{\bG}}+\operatorname{symblockdiag}(\bC \bar{\bm{H}}_{\bar{\bG}} \bar{\bG}^\top+\bC\bar{\bG} \bar{\bm{H}}_{\bar{\bG}}^\top )\bar{\bG}\right)\notag
\end{aligned}
\]
and $\operatorname{Proj}_{\operatorname{T}_{\bar{\bG}}\odn }\left(\operatorname{symblockdiag}(\bC \bar{\bm{H}}_{\bar{\bG}} \bar{\bG}^\top+\bC\bar{\bG} \bar{\bm{H}}_{\bar{\bG}}^\top )\bar{\bG}\right)=\bm{0}$,
it follows from \eqref{Hess-def} that
\begin{align*}
\overline{\operatorname{Hess} f([\bG])\left[\bm{H}_{[\bG]}\right]}_{\bar{\bG}}
&= \operatorname{Proj}_{\mathcal{H}_{\bar{\bG}}}\left(\operatorname{Proj}_{\operatorname{T}_{\bar{\bG}}\odn}\left(-2S(\bar{\bG})\bar{\bm{H}}_{\bar{\bG}}\right) \right)\notag\\
&=\left(\bm{I}_{nd}-\frac{1}{n} \bar{\bG}\bar{\bG}^\top  \right)\left(\operatorname{Proj}_{\operatorname{T}_{\bar{\bG}}\odn }\left(-2S(\bar{\bG})\bar{\bm{H}}_{\bar{\bG}}\right) \right).
\end{align*}
The proof is complete.
\end{proof}

\section{Improved Estimation: from Average to Worst-Case}
\label{section-improvedest}

In this section, we investigate the deterministic estimation performance of the estimator \eqref{Problem0} (also \eqref{Problem}) for the ground truth $\bG^\star\in\ogn$ from the existing $\ell_2$ distance to $\ell_\infty$ distance. The derived $\ell_\infty$ distance bound is similar to the previous one under $\ell_2$ distance when it transferred to $\ell_2$ distance under certain statistical model. However, it controls elementwise estimation error which improves the result from the average to worst-case scenario.  

Before presenting the main results, 
we first introduce the following $\ell_2$ estimation error which is similar to the result in \cite[Lemma 4.1 and 4.2]{zhu2021orthogonal}.



\begin{lemma}[$\ell_{2}$-Estimation Error]\label{dist-GstarGhat}
   Let $\bG \in \mathbb{R}^{n d \times d}$ be such that $\bG^{\top} \bG=n \bm{I}_{d}$ and $\operatorname{tr}(\bG^{\top} \bC \bG) \geq$ $\operatorname{tr}(\bG^{\star \top} \bC \bG^{\star})$. 
   Then it follows that  
  \begin{equation}\label{dist-GstarGhat-ineq}
  \operatorname{d}_F([\bG],[\bG^\star])
  \leq   \frac{4\sqrt{d}\|\bm{\Delta}\|}{\sqrt{n}}.
  \end{equation}
  Moreover, all singular values of $\bG^{\star\top} \bG$ satisfy
  \begin{align} \label{eq: Gstarhatsingularvalue}
    n-\frac{8d\|\bm{\Delta}\|^2}{n} \leq \sigma_{l}(\bG^{\star\top} \bG) \leq n, \quad \text{for each}\ l\in[d],
  \end{align}
and the smallest singular value of $\bC_{i,:}\bG$ for $i\in[n]$ satisfies
\begin{align}\label{singular-lowbd}
\sigma_{\min }(\bC_{i,:}\bG)\ge n-\frac{8d\|\bm{\Delta}\|^2}{n}-\|\bm{\Delta} \bG\|_{\infty}.
\end{align}
  \end{lemma}
\begin{proof}
From the assumption $\operatorname{tr}(\bG^{\top} \bC \bG) \geq \operatorname{tr}(\bG^{\star\top} \bC \bG^{\star})$ and the generative model $\bC=\bG^\star \bG^{\star\top}+\bm{\Delta}$, we know that
\begin{align}\label{dist-GstarGhat-key1}
n^{2} d-\operatorname{tr}(\bG^{\top} \bG^{\star}\bG^{\star\top} \bG) 
&= \operatorname{tr}(\bG^{\star\top} \bG^{\star}\bG^{\star\top} \bG^{\star})-\operatorname{tr}(\bG^{\top} \bG^{\star}\bG^{\star\top} \bG)\notag\\
&\leq \operatorname{tr}(\bG^{\top} \bm{\Delta} \bG)-\operatorname{tr}(\bG^{\star\top} \bm{\Delta} \bG^{\star})\notag\\
&=\operatorname{tr}\left(\left(\bG-\bG^{\star}\right)^{\top} \bm{\Delta}\left(\bG+\bG^{\star}\right)\right).
\end{align}
On the other hand, we have
\begin{equation}\label{dist-GstarGhat-key2}
  \operatorname{tr}(\bG^{\top} \bG^{\star}\bG^{\star\top} \bG)\leq \Vert \bG^{\top} \bG^{\star}\Vert\cdot\Vert \bG^{\top} \bG^{\star}\Vert_*\leq n \Vert \bG^{\top} \bG^{\star}\Vert_*.
\end{equation}
Without of loss of generality, we assume $\bG$ satisfies $\operatorname{tr}(\bG^{\top} \bG^{\star})=\|\bG^{\top} \bG^{\star}\|_*$, and then it follows from \eqref{dist-GstarGhat-key1} and \eqref{dist-GstarGhat-key2} that
\[
\begin{aligned}
\frac{1}{2}\operatorname{d}_F([\bG], [\bG^{\star}])^{2}=n d-\|\bG^{\top} \bG^{\star}\|_* 
& \leq \frac{n^{2} d-\operatorname{tr}(\bG^{\top} \bG^{\star}\bG^{\star\top} \bG)}{n} \\
&\leq\frac{1}{n} \operatorname{tr}\left(\left(\bG-\bG^{\star}\right)^{\top} \bm{\Delta}\left(\bG+\bG^{\star}\right)\right) \\
& \leq \frac{1}{n}\|\bm{\Delta}\|\|\bG-\bG^{\star}\|_{F}\|\bG+\bG^{\star}\|_{F} \\
&=\frac{2 \sqrt{n d}}{n}\|\bm{\Delta}\|\cdot \operatorname{d}_F([\bG], [\bG^{\star}]),
\end{aligned}
\]
which implies that \eqref{dist-GstarGhat-ineq} holds. Next, from \eqref{dist-GstarGhat-ineq} we know that
\begin{align}\label{key-from-4.1}
nd-\Vert \bG^{\star\top} \bG\Vert_*=\frac{1}{2}\operatorname{d}_F([\bG^\star],[\bG])^2
&\leq   \frac{8d\|\bm{\Delta}\|^2}{n}.
\end{align}
Since $\|\bG^{\star\top} \bG\| \leq n$ implies $0 \leq \sigma_{l}(\bG^{\star\top} \bG) \leq n$ for $l\in[d]$, combined with \eqref{key-from-4.1}  one has
\[
\sum_{l=1}^{d}\left(n-\sigma_{l}(\bG^{\star\top} \bG)\right)=nd-\Vert \bG^{\star\top} \bG\Vert_*\leq   \frac{8d\|\bm{\Delta}\|^2}{n},
\]
which implies that $\sigma_{\min}(\bG^{\star\top} \bG) \geq n-\frac{8d\|\bm{\Delta}\|^2}{n}$.
This together with the definition $\bC_{i,:} \bG=\bG_i^\star \bG^{\star\top}\bG+\bm{\Delta}_{i,:} \bG$ shows that
\[
\begin{aligned}
\sigma_{\min }(\bC_{i,:}\bG) & \geq  \sigma_{\min}(\bG^{\star\top}\bG)-\|(\bm{\Delta} \bG)_i \| \ge n-\frac{8d\|\bm{\Delta}\|^2}{n}-\|\bm{\Delta} \bG\|_{\infty}.
\end{aligned}
\]
The proof is complete.
\end{proof}

Recall that $\widehat{\bG}$ is one of the optimal solutions of problem \eqref{Problem0} and let $\bm{\Phi}$ be the matrix of top $d$ eigenvectors of $\bC$ with $\bm{\Phi}^\top \bm{\Phi} = n\bm{I}_d$. The following proposition illustrates that the improved estimation error of least-squares estimator which can be controlled by not only  $\ell_2$ distance $\operatorname{d}_F(\cdot,\cdot)$ but also $\ell_{\infty}$ distance $\operatorname{d}_{\infty}(\cdot,\cdot)$.
\begin{prop}[$\ell_{\infty}$-Estimation Error]\label{lemma-distinf-ineq}
Suppose that $\|\bm{\Delta}\|\leq \frac{n}{6d^{1/2}}$ and $\bG=\widehat{\bG}$ or $\bm{\Phi}$. 
Then for $n\ge2$, it follows that
  \begin{equation}\label{distinf-GstarGhat-ineq}
  \operatorname{d}_{\infty}([\bG],[\bG^\star])
  \leq\|\bG\bQ^\star-\bG^\star\|_{\infty}
  \leq
  \frac{8\left\|\bm{\Delta} \bG^\star\right\|_{\infty}}{n},
  \end{equation}
where  $\bQ^\star\in\od$ satisfy $\operatorname{d}_F([\bG],[\bG^\star])=\Vert \bG-\bG^\star \bQ^\star\Vert_F$.
\end{prop}

\begin{proof}
Fix $t\in\{1,\ldots,n\}$ and let $\tilde{\bG} \in \mathbb{R}^{nd \times d}$ satisfy 
\[ 
  \tilde{\bG}_{i}=\left\{
  \begin{array}{c@{\quad}l}
  \displaystyle \bG^\star_{i}, & \textrm{if } i=t, \\
  \noalign{\smallskip}
  \bG_{i}, & \textrm{otherwise}.
  \end{array}
  \right.
\]
Since $\bG=\widehat{\bG}$ (resp. $\bG=\bm{\Phi}$) is the optimum of the optimization problem $\max_{\bG\in\odn} \operatorname{tr}(\bG^\top \bC\bG)$ (resp. $\max_{\bm{X}^\top\bm{X}=n\bm{I}_d} \operatorname{tr}(\bm{X}^\top \bC\bm{X})$), we know that $\operatorname{tr}(\bG^\top \bC\bG) \geq \operatorname{tr}(\tilde{\bG}^\top \bC \tilde{\bG})$, which is equivalent to 
\[
\operatorname{tr}\left((\bG\bQ^\star-\tilde{\bG})^\top \bC \bG\bQ^\star\right)+\operatorname{tr}\left(\bQ^{\star\top}\bG^\top \bC(\bG\bQ^\star-\tilde{\bG})\right)-\operatorname{tr}\left((\bG\bQ^\star-\tilde{\bG})^\top \bC(\bG\bQ^\star-\tilde{\bG})\right) \geq 0.
\] 
This together with the definition of $\tilde{\bG}$, $\bG$ and $\bC=\bG^\star \bG^{\star\top}+\bm{\Delta}$ implies that
\begin{equation*}
\begin{aligned}
& 2 \operatorname{tr}\left((\bG_{t}\bQ^\star-\bG^\star_{t})^\top \bG^\star_{t} \bG^{\star\top} \bG\bQ^\star\right) +2 \operatorname{tr}\left((\bG_{t}\bQ^\star-\bG^\star_{t})^\top\bm{\Delta}_{t,:} \bG\bQ^\star\right)\\
& -\operatorname{tr}\left((\bG_{t}\bQ^\star-\bG^\star_{t})^\top \bC_{tt}(\bG_{t}\bQ^\star-\bG^\star_{t})\right)\\
=&\ \operatorname{tr}\left((\bG_{t}\bQ^\star-\bG^\star_{t})^\top \bC_{t,:} \bG\bQ^\star\right) + \operatorname{tr}\left( \bQ^{\star\top}\bG^\top \bC_{t,:}^\top(\bG_{t}\bQ^\star-\bG^\star_{t})\right)\\
&-\operatorname{tr}\left((\bG_{t}\bQ^\star-\bG^\star_{t})^\top \bC_{tt}(\bG_{t}\bQ^\star-\bG^\star_{t})\right)
\geq 0,
\end{aligned}
\end{equation*}
and consequently
\begin{equation}\label{conponentwise-key0}
\begin{aligned}
  \operatorname{tr}\left((\bG_{t}\bQ^\star-\bG^\star_{t})^\top\bm{\Delta}_{t,:} \bG\bQ^\star\right)
&\ge  \operatorname{tr}\left((\bG^\star_{t}-\bG_{t}\bQ^\star)^\top \bG^\star_{t} \bG^{\star\top} \bG\bQ^\star\right)\\
&\quad + \frac{1}{2}\operatorname{tr}\left((\bG_{t}\bQ^\star-\bG^\star_{t})^\top \bC_{tt}(\bG_{t}\bQ^\star-\bG^\star_{t})\right)
\end{aligned}
\end{equation}
Note that $\bG^{\star\top} \bG \bQ^\star$ is symmetric and satisfies $\bG^{\star\top} \bG \bQ^\star\succeq \bm{0}$, 
and then it follows that
\begin{equation}\label{conponentwise-key1}
\begin{aligned}
&\operatorname{tr} \left((\bG^\star_{t}-\bG_{t}\bQ^\star)^\top \bG^\star_{t} \bG^{\star\top} \bG\bQ^\star\right)\\
=&\ \operatorname{tr} \left( \bG^\star_{t} \bG^{\star\top} \bG\bQ^\star\right)-\frac{1}{2}\operatorname{tr} \left(\bQ^{\star\top}\bG_{t}^\top \bG^\star_{t} \bG^{\star\top} \bG\bQ^\star\right)-\frac{1}{2}\operatorname{tr} \left( \bG^{\star\top}_{t}\bG_{t}\bQ^{\star} \bG^{\star\top} \bG\bQ^\star\right)\\
=&\ \frac{1}{2} \operatorname{tr}\left((\bG_{t}\bQ^\star-\bG_{t}^\star)^{\top}(\bG_{t}\bQ^\star-\bG^\star_{t})\bG^{\star\top} \bG \bQ^\star\right) \\
\geq&\ \frac{1}{2}\sigma_{\min}(\bG^{\star\top} \bG)\cdot\|\bG_{t}\bQ^\star-\bG^\star_{t}\|_{F}^{2}.
\end{aligned}
\end{equation}
On the other hand, one has that
\begin{equation}\label{conponentwise-key3}
\operatorname{tr}\left((\bG_{t}\bQ^\star-\bG^\star_{t})^\top\bm{\Delta}_{t,:} \bG\bQ^\star\right) \leq \|\bm{\Delta}_{t,:} \bG\|_{F}\|\bG_{t}\bQ^\star-\bG^\star_{t}\|_{F} ,
\end{equation}
and also 
\begin{equation}\label{conponentwise-key2}
\operatorname{tr}\left((\bG_{t}\bQ^\star-\bG^\star_{t})^\top \bC_{tt}(\bG_{t}\bQ^\star-\bG^\star_{t})\right) \geq-\|\bC_{tt}\|\cdot\|\bG_{t}\bQ^\star-\bG^\star_{t}\|_{F}^{2} \geq-\|\bG_{t}\bQ^\star-\bG^\star_{t}\|_{F}^{2}
\end{equation}
(noting that $\|\bm{\Delta}_{tt}\|=0$).
Thus, combining \eqref{conponentwise-key0} with \eqref{conponentwise-key1}, \eqref{conponentwise-key3} and \eqref{conponentwise-key2}, it concludes that
\[
2\|\bm{\Delta}_{t,:} \bG\|_{F}\|\bG_{t}\bQ^\star-\bG^\star_{t}\|_{F} \geq(\sigma_{\min}(\bG^{\star\top} \bG)-1)\cdot\|\bG_{t}\bQ^\star-\bG^\star_{t}\|_{F}^{2} .
\]
Consequently, from \eqref{eq: Gstarhatsingularvalue} and the arbitrariness of $t$, it follows that
\begin{equation}\label{conponentwise-con}
\begin{aligned}
\operatorname{d}_{\infty}([\bG],[\bG^\star]) \leq \|\bG\bQ^\star-\bG^\star\|_{\infty}
  &= \max_{i\in[n]}\|\bG_i\bQ^\star-\bG^{\star}_i\|_F\\
  &\leq\frac{2\max_{i \in [n]}\left\|\bm{\Delta}_{i,:} \bG^\star\right\|_{F}}{n-\frac{8d\|\bm{\Delta}\|^2}{n}-1}\leq\frac{8\left\|\bm{\Delta} \bG^\star\right\|_{\infty}}{n},
  \end{aligned}
  \end{equation}
where the last inequality is from $\|\bm{\Delta}\|\leq \frac{n}{6d^{1/2}}$ when $n\ge2$. The proof is complete. 
\end{proof}

\begin{remark}\label{rmk:information-limit}
Under the standard Gaussian noise setting (i.e., the noise matrices take the form $\bm{\Delta}_{ij} = \sigma \bm{W}_{ij}$ for $i<j$, where $\bm{W}_{ij} \in \mathbb{R}^{d \times d}$ has independent standard Gaussian entries and $\sigma>0$ is the noise level),
we know that with high probability
\[
\|\bm{\Delta} \bG^\star\|_{\infty}\leq \sqrt{d}\cdot\max _{1 \leq i \leq n}\|\bm{\Delta}_{i,:} \bG^\star\| \lesssim \sigma \sqrt{n}(d+\sqrt{d\log n})
\]
by taking the union bound over all $1 \leq i \leq n$ on
$\bm{\Delta}_{i,:} \bG^\star=\sigma \sum_{j \neq i} \boldsymbol{W}_{i j}$ with \cite[Theorem 4.4.5]{vershynin2018high}. 
Thus, $\sigma\lesssim \sqrt{n/d^2\log n}$ would satisfy the requirement in Proposition \ref{lemma-distinf-ineq} for exact recovery. On the other hand, it is believed that $\sigma\lesssim \sqrt{n / d}$ (i.e., $\|\bm{\Delta}\|\lesssim n$ since $\|\boldsymbol{\Delta}\| \lesssim  \sigma \sqrt{n d}$ with high probability \cite{bandeira2017tightness}) is the threshold above which is information-theoretically impossible to exactly recover the signals of the generative model $\bC=\bG\bG^{\top}+\bm{\Delta}$ \cite{jung2020weak}. 
Hence, the result in Proposition \ref{lemma-distinf-ineq} is near-optimal in $n$ and
only differs from the information-theoretic threshold by a sub-optimal factor of $\sqrt{d}$.
\end{remark}

\begin{remark}
With the assumption 
that $\|\bm{\Delta}\|\lesssim n/\sqrt{d}$ and $\|\bm{\Delta} \bG^\star\|_{\infty}\lesssim n$ (i.e., $\sigma\lesssim \sqrt{n/d^2\log n}$ under standard Gaussian noise setting from Remark \ref{rmk:information-limit}), 
we know from Proposition \ref{lemma-distinf-ineq} that $\operatorname{d}_{\infty}([\widehat{\bG}],[\bG^\star]) = \mathcal{O}(1)$. Then the least-squares estimator $\widehat{\bG}$ staying in the same/totally reverse connected component with $\bG^\star$ (otherwise $\operatorname{d}_{\infty}([\widehat{\bG}], [\bG^\star])\geq\sqrt{2}$; cf. \cite[p. 101]{rosen2019se}). This 
makes the relaxation of problem \eqref{Problem0} for \eqref{ProblemR} tight and the results of 
the spectral method \cite{ling2022near}, the SDR \cite{won2022orthogonal,ling2022solving}, the Burer-Monteiro factorization \cite{boumal2016non,ling2022solving} and the GPM \cite{liu2020unified,zhu2021orthogonal,ling2022improved}
developed for solving orthogonal group synchronization
can be directly applied to solve the rotation synchronization problem under near-optimal noise level for exact recovery.
\end{remark}

\begin{figure}[h]
    \centering
    \includegraphics[width=0.65\textwidth]{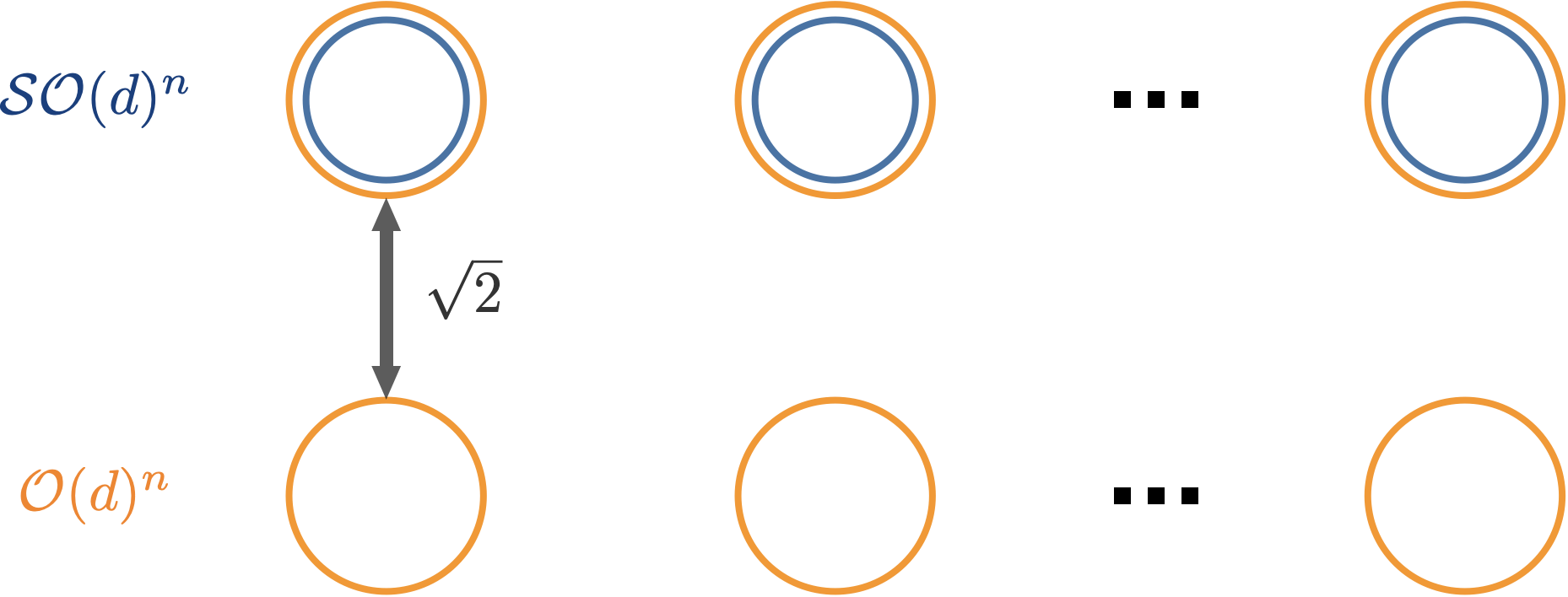}
    \caption{Minimal distance between components \& Tightness of \eqref{Problem0} for \eqref{ProblemR}.}
    \label{fig:samecomponent}
\end{figure}

\section{Landscape Analysis via Quotient Manifold}
\label{section-landscape}
In this section, we focus on the landscape analysis of orthogonal synchronization problem through the established quotient approach. 
We will show the local strongly concave property around 
the global maximizer $[\widehat{\bG}]$ of the problem \eqref{Problem}  (i.e., $\widehat{\bG}\in\odn$ is a global maximizer of the original problem \eqref{Problem0}). Then, as a byproduct, we derive the local error bound property of the original problem \eqref{Problem0}.

First, we introduce the main result of this section that problem \eqref{Problem} is locally geodesically strongly concave around the global maximizer under certain additive noise.
\begin{theorem}[Local Geodesic Strong Concavity]\label{hess-positive}
Suppose that $\|\bm{\Delta}\|\le \frac{n^{3/4}}{20d^{1/2}}$ and $\|\bm{\Delta} \bG^{\star}\|_{\infty}\le \frac{n}{20}$. Then with
\begin{equation}\label{dist-rhoF-rhoinfinity}
\rho_F=\frac{1}{10}\min\left\{\sqrt{n},\frac{n}{\|\bm{\Delta}\|}\right\}\quad \text{and}\quad  \rho_{\infty}=\frac{1}{4}
\end{equation}
such that for any $\bG\in\odn$ satisfying $\operatorname{d}_F([\bG], [\widehat{\bG}])\leq \rho_F$ and $\|\bG\widehat{\bQ}- \widehat{\bG}\|_\infty\leq \rho_{\infty}$, it follows that 
\[
-\left\langle\Hess f([\bG])[\bm{H}_{[\bG]}], \bm{H}_{[\bG]}\right\rangle\ge \frac{n}{5}\cdot\langle \bm{H}_{[\bG]}, \bm{H}_{[\bG]}\rangle>0,\quad \text{for all}\ \bm{H}_{[\bG]} \in \operatorname{T}_{[\bG]} \mathcal{Q} \backslash\left\{\bm{0}_{[\bG]}\right\},
\]
where $\widehat{\bQ}\in\od$ satisfies $\operatorname{d}_F([\bG],[\widehat{\bG}])=\Vert \bG \widehat{\bQ}-\widehat{\bG}\Vert_F$.
\end{theorem}

\begin{proof}
Since for any $[\bG]\in\mathcal{Q}$ and $\bm{H}_{[\bG]} \in \operatorname{T}_{[\bG]} \mathcal{Q} \backslash\{\bm{0}_{[\bG]}\}$ one has that
\[
\left\langle\Hess f([\bG])[\bm{H}_{[\bG]}], \bm{H}_{[\bG]}\right\rangle=\left\langle\overline{\Hess f([\bG])[\bm{H}_{[\bG]}]}_{\bG}, \bar{\bm{H}}_{\bG}\right\rangle,
\]
then it suffices to prove that
\[
-\left\langle\overline{\Hess f([\bG])[\bm{H}_{[\bG]}]}_{\bG}, \bar{\bm{H}}_{\bG}\right\rangle>0, \quad \text{for all}\ \bar{\bm{H}}_{\bG} \in \mathcal{\bm{H}}_{\bG} \backslash \{\bm{0}_{\bG}\},
\]
where from \eqref{Hess-cal} and $\bar{\bm{H}}_{\bG}:=[\ldots;\bG_i\bm{E}_i;\ldots]$ with $\sum_{i=1}^n \bm{E}_i=\bm{0}$ we know that
\begin{align}\label{Hessian-cal}
&-\frac{1}{2}\left\langle\overline{\Hess f([\bG])[\bm{H}_{[\bG]}]}_{\bG}, \bar{\bm{H}}_{\bG}\right\rangle\notag\\
=&\ \operatorname{tr}\left(\bar{\bm{H}}_{\bG}^\top \left(\bm{I}_{nd}-\frac{1}{n} \bG\bG^\top  \right)\left(\operatorname{Proj}_{\operatorname{T}_{\bG}\odn}(S(\bG)\bar{\bm{H}}_{\bG}) \right)        \right)\notag\\
=&\ \operatorname{tr}\left(\bar{\bm{H}}_{\bG}^\top\left(\bm{I}_{nd}-\frac{1}{n}  \bG\bG^\top  \right)\left(\operatorname{Diag}\left(\left[\ldots;\bC_{i,:} \bG\bG_i^\top;\ldots\right]\right)-\bC\right)\bar{\bm{H}}_{\bG}\right)\notag\\
=&\ \operatorname{tr}\left(\left(\left[\ldots,\bm{E}_i^\top \bG_i^\top,\ldots\right]-\frac{1}{n}\sum_i \bm{E}_i^\top \bG^\top\right)\left(\left[\ldots;\bC_{i,:} \bG \bm{E}_i;\ldots\right]-\left[\ldots;\sum_j \bC_{ij} \bG_j \bm{E}_j;\ldots\right]\right)\right)\notag\\
=&\ \operatorname{tr}\left([\ldots,\bm{E}_i^\top \bG_i^\top,\ldots]\left([\ldots;\bC_{i,:} \bG \bm{E}_i;\ldots]-\left[\ldots;\sum_j \bC_{ij} \bG_j \bm{E}_j;\ldots\right]\right)\right)\notag\\
=&\ \operatorname{tr}\left(\sum_i \bm{E}_i^\top \bG_i^\top \bC_{i,:} \bG \bm{E}_i-\sum_i\sum_j \bm{E}_i^\top \bG_i^\top   \bC_{ij} \bG_j\bm{E}_j\right).
\end{align}
Let $\bQ^\star\in\od$ satisfy $\operatorname{d}_F([\bG^\star],[\bG])=\Vert \bG^\star \bQ^\star-\bG\Vert_F$ for any $\bG\in\od$, and from $\sum_{i=1}^n \bm{E}_i=\bm{0}$ we know that $\bG^\top [\ldots;\bG_i \bm{E}_i;\ldots]=\bm{0}$, which 
implies that
\begin{equation}\label{Hess-cal-key}
\left\|\bG^{\star\top}[\ldots;\bG_i \bm{E}_i;\ldots]\right\|_F^2=\left\|[\ldots;\bG_i \bm{E}_i;\ldots]^{\top}(\bG^{\star}-\bG \bQ^{\star})\right\|_{F}^{2} \leq \operatorname{d}_F([\bG^{\star}], [\bG])^{2} \cdot\sum_i\|\bm{E}_i\|_F^2.
\end{equation}
Thus, from \eqref{Hessian-cal} with $\bC=\bG^\star \bG^{\star\top}+\bm{\Delta}$,
we derive that for any $[\bG]\in\mathcal{Q}$ and $\bG\in \pi^{-1}([\bG])$,
  \begin{align}\label{pd-ineq}
    &-\frac{1}{2}\left\langle\overline{\Hess f([\bG])[\bm{H}_{[\bG]}]}_{\bG}, \bar{\bm{H}}_{\bG}\right\rangle\notag\\
    =&\ \operatorname{tr}\left(\sum_i \bm{E}_i^\top \bG_i^\top \bC_{i,:} \bG \bG_i^\top \bG_i \bm{E}_i -\sum_i \sum_j  \bm{E}_i^\top \bG_i^\top \bG_i^{\star}\bG_j^{\star\top}\bG_j\bm{E}_j-\sum_i\sum_j   \bm{E}_i^\top \bG_i^\top \bm{\Delta}_{ij}\bG_j\bm{E}_j \right)\notag\\
  =&\  \operatorname{tr}\left(\sum_i \bm{E}_i^\top \bG_i^\top \bC_{i,:} \widehat{\bG} \widehat{\bG}_i^\top \bG_i \bm{E}_i\right)+\operatorname{tr}\left(\sum_i \bm{E}_i^\top \bG_i^\top (\bC_{i,:} \bG \bG_i^\top-\bC_{i,:} \widehat{\bG} \widehat{\bG}_i^\top) \bG_i \bm{E}_i\right)\notag\\
  & - \left\|\bG^{\star\top}[\ldots;\bG_i \bm{E}_i;\ldots]\right\|_F^2 -\operatorname{tr}\left([\ldots;\bG_i \bm{E}_i;\ldots]^\top\bm{\Delta}[\ldots;\bG_i \bm{E}_i;\ldots]\right)\notag\\
    \ge&\ \left(\min_{i\in[n]}\left\{\sigma_{\min}(\bC_{i,:}\widehat{\bG}\widehat{\bG}_i^\top)\right\}-\max_{i\in[n]}\|\bC_{i,:}(\bG\widehat{\bQ}-\widehat{\bG})\|_F-\max_{i\in[n]}\|\bC_{i,:}\widehat{\bG}(\bG_i\widehat{\bQ}-\widehat{\bG}_i)^\top\|_F\right)\sum_i\Vert \bm{E}_i\Vert_F^2 \notag\\
    & - \operatorname{d}_F([\bG^{\star}], [\bG])^{2} \cdot\sum_i\|\bm{E}_i\|_F^2 -  \|\bm{\Delta}\|\cdot\sum_i\|\bm{E}_i\|_F^2,
    \end{align}
where $\widehat{\bG}\in\od$ is a global maximizer of problem \eqref{Problem0} indicating that $\bC_{i,:}\widehat{\bG}\widehat{\bG}_i^\top\succeq\bm{0}$ by \cite[Lemma 3.6]{zhu2021orthogonal}, and the inequality is from \eqref{Hess-cal-key}, Cauchy-Schwarz inequality and also
\begin{align*}
&\operatorname{tr}\left(\sum_i \bm{E}_i^\top \bG_i^\top (\bC_{i,:} \bG \bG_i^\top-\bC_{i,:} \widehat{\bG} \widehat{\bG}_i^\top) \bG_i \bm{E}_i\right)\notag\\
\ge&\ -  \sum_i \|\bC_{i,:} \bG \bG_i^\top-\bC_{i,:} \widehat{\bG}\widehat{\bQ}^\top \bG_i^\top+\bC_{i,:} \widehat{\bG}\widehat{\bQ}^\top \bG_i^\top-\bC_{i,:} \widehat{\bG} \widehat{\bG}_i^\top\|_F\|\bm{E}_i\|_F^2\notag\\
\ge&\ -\left(\max_{i\in[n]}\|\bC_{i,:}(\bG\widehat{\bQ}-\widehat{\bG})\|_F+\max_{i\in[n]}\|\bC_{i,:}\widehat{\bG}(\bG_i\widehat{\bQ}-\widehat{\bG}_i)^\top\|_F\right)\sum_i\Vert \bm{E}_i\Vert_F^2.
\end{align*}

To proceed, in the following part we will analyze each term in \eqref{pd-ineq}. First, from  \eqref{singular-lowbd} one has that
\begin{equation}\label{pd-ineq-key-0}
\min_{i\in[n]}\left\{\sigma_{\min}(\bC_{i,:}\widehat{\bG}\widehat{\bG}_i^\top)\right\}=\min_{i\in[n]}\left\{\sigma_{\min}(\bC_{i,:}\widehat{\bG})\right\}\ge n-\frac{8d\|\bm{\Delta}\|^2}{n}-\|\bm{\Delta} \widehat{\bG}\|_{\infty}.
\end{equation}
Next, since we know from \eqref{dist-GstarGhat-ineq} that
\begin{align}
  \Vert \bG^{\star\top}(\bG\widehat{\bQ} - \widehat{\bG})\Vert_F
  &\le \Vert (\bG^\star-\widehat{\bG}\widehat{\bQ}^\star)^\top (\bG\widehat{\bQ} - \widehat{\bG})\Vert_F +\Vert \widehat{\bG}^\top (\bG\widehat{\bQ} - \widehat{\bG})\Vert_F\notag\\
  &\le \Vert \bG^\star-\widehat{\bG}\widehat{\bQ}^\star\Vert_F\Vert \bG\widehat{\bQ} - \widehat{\bG}\Vert_F + \sqrt{\Vert \widehat{\bG}^\top \bG\|_F^2-2n\|\widehat{\bG}^\top \bG\|_*+ n^2 d}\notag\\
   &= \operatorname{d}_F([\bG],[\bG^\star])\cdot\operatorname{d}_F([\bG],[\widehat{\bG}]) + \sqrt{\sum_{i}^d\left(n-\sigma_i(\widehat{\bG}^\top \bG)\right)^2}\notag\\
   & \leq \frac{4\sqrt{d} \left\| \bm{\Delta} \right\| }{\sqrt{n}}\cdot \operatorname{d}_F([\bG],[\widehat{\bG}])+\sqrt{\left(nd-\|\widehat{\bG}^\top \bG\|_*\right)^2}\notag\\
  & = \frac{4\sqrt{d} \left\| \bm{\Delta} \right\| }{\sqrt{n}}\cdot \operatorname{d}_F([\bG],[\widehat{\bG}])+\frac{1}{2}\operatorname{d}_F([\bG],[\widehat{\bG}])^2,\notag
  \end{align} 
where $\widehat{\bQ}^\star\in\od$ satisfy $\operatorname{d}_F([\bG^\star],[\widehat{\bG}])=\Vert \bG^\star \widehat{\bQ}^\star-\widehat{\bG}\Vert_F$, 
it follows that
\begin{align}\label{pd-ineq-key}
  \max_{i\in[n]}\|\bC_{i,:}(\bG\widehat{\bQ}-\widehat{\bG})\|_F
  &\leq\max_{i\in[n]}\|\bG_i^\star \bG^{\star\top}(\bG\widehat{\bQ}-\widehat{\bG})\|_F+\max_{i\in[n]}\|\bm{\Delta}_i(\bG\widehat{\bQ}-\widehat{\bG})\|_F\notag\\
  &\leq\| \bG^{\star\top}(\bG\widehat{\bQ}-\widehat{\bG})\|_F+\|\bm{\Delta}\|\cdot \operatorname{d}_F([\bG],[\widehat{\bG}])\notag\\
  &\leq \frac{4\sqrt{d}\|\bm{\Delta}\|}{\sqrt{n}}\cdot\operatorname{d}_F([\bG], [\widehat{\bG}])+\frac{1}{2}\operatorname{d}_F([\bG], [\widehat{\bG}])^2 +\|\bm{\Delta}\|\cdot\operatorname{d}_F([\bG], [\widehat{\bG}])\notag\\
  &\leq\left(\frac{1}{2}\operatorname{d}_F([\bG], [\widehat{\bG}])+\frac{4\sqrt{d}\|\bm{\Delta}\|}{\sqrt{n}}+\|\bm{\Delta}\|\right) \cdot \operatorname{d}_F([\bG], [\widehat{\bG}]).
\end{align}
On the other hand, we know that
\begin{align}\label{pd-ineq-key-2}
\max_{i\in[n]}\|\bC_{i,:}\widehat{\bG}(\bG_i\widehat{\bQ}-\widehat{\bG}_i)^\top\|_F
&\leq \max_{i\in[n]}\|\bC_{i,:}\widehat{\bG}\|\cdot\max_{i\in[n]}\|\bG_i\widehat{\bQ}-\widehat{\bG}_i\|_F\notag\\
&\leq \left(\|\bG^{\star\top}\widehat{\bG}\|+\|\bm{\Delta}\widehat{\bG}\|_{\infty}\right)\cdot\|\bG\widehat{\bQ}- \widehat{\bG}\|_\infty\notag\\
&\leq \left(n+\|\bm{\Delta}\widehat{\bG}\|_{\infty}\right)\cdot \|\bG\widehat{\bQ}- \widehat{\bG}\|_\infty.
\end{align}
Thus, combining \eqref{pd-ineq} with \eqref{pd-ineq-key-0}, \eqref{pd-ineq-key} and \eqref{pd-ineq-key-2} we know that
\begin{align}\label{pd-conclus}
  &-\frac{1}{2}\left\langle\Hess f([\bG])[\bm{H}_{[\bG]}], \bm{H}_{[\bG]}\right\rangle\notag\\
  \ge&\  \Bigg( n-\frac{8d\|\bm{\Delta}\|^2}{n}-\|\bm{\Delta} \widehat{\bG}\|_{\infty}-\Vert\bm{\Delta}\Vert-\left(\frac{1}{2}\operatorname{d}_F([\bG], [\widehat{\bG}])+\frac{4\sqrt{d}\|\bm{\Delta}\|}{\sqrt{n}}+\|\bm{\Delta}\|\right) \cdot \operatorname{d}_F([\bG], [\widehat{\bG}])\notag\\
  & -\operatorname{d}_F([\bG],[\bG^{\star}])^{2}-(n+\bm{\|\Delta}\widehat{\bG}\|_{\infty})\cdot\|\bG\widehat{\bQ}- \widehat{\bG}\|_\infty\Bigg) \cdot\sum_i\Vert \bm{E}_i\Vert_F^2.
  \end{align}
Consequently, the inequality \eqref{pd-conclus} 
together with 
\begin{align*}
  \operatorname{d}_F([\bG],[\bG^\star]) \leq \operatorname{d}_F([\bG],[\widehat{\bG}]) + \operatorname{d}_F([\widehat{\bG}], [\bG^\star]) \leq \operatorname{d}_F([\bG],[\widehat{\bG}]) + \frac{4\sqrt{d}\|\bm{\Delta}\|}{\sqrt{n}},
\end{align*}
 and also
\begin{equation}\label{hat-star-trans}
\begin{aligned}
\|\bm{\Delta} \widehat{\bG}\|_{\infty} \leq \| \bm{\Delta} (\widehat{\bG} \widehat{\bQ}^{\star}-\bG^{\star})\|_{\infty}+\left\|\bm{\Delta} \bG^{\star}\right\|_{\infty} 
&\leq \| \bm{\Delta}\|\cdot \operatorname{d}_F([\widehat{\bG}], [\bG^{\star}])+\left\|\bm{\Delta} \bG^{\star}\right\|_{\infty} \\
&\leq \frac{4 \sqrt{d}\| \bm{\Delta}\|^2}{\sqrt{n}}+\left\|\bm{\Delta} \bG^{\star}\right\|_{\infty}
\end{aligned}
\end{equation}
by \eqref{dist-GstarGhat-ineq} completes the proof with the given assumptions.
\end{proof}

The following proposition shows that under the same noise level as in Theorem \ref{hess-positive}, the ground truth $[\bG^\star]\in\mathcal{Q}$ falls in the region that
the function $f$ is geodesically strongly concave.
\begin{prop}\label{prop-sol-in}
Suppose that $\|\bm{\Delta}\|\le \frac{n^{3/4}}{80d^{1/2}}$ and $\|\bm{\Delta} \bG^{\star}\|_{\infty}\le \frac{n}{40}$ with $n\ge2$. 
Then for any global maximizer $\widehat{\bG}\in\odn$ of \eqref{Problem0}, it satisfies 
\[
\operatorname{d}_F([\widehat{\bG}], [\bG^\star])\leq \frac{1}{2}\rho_F\quad \text{and}\quad \|\widehat{\bG}\widehat{\bQ}^\star- \bG^\star\|_\infty\leq \frac{1}{2}\rho_{\infty},
\] 
where $\rho_F$, $\rho_{\infty}$ are given by Theorem \ref{hess-positive}.
\end{prop}
\begin{proof}
From \eqref{dist-GstarGhat-ineq} and \eqref{conponentwise-con} we know that
\begin{equation}\label{key-sol-in}
  \operatorname{d}_F([\widehat{\bG}],[\bG^\star])\leq   \frac{4\sqrt{d}\|\bm{\Delta}\|}{\sqrt{n}}\quad\text{and}\quad \|\widehat{\bG}\widehat{\bQ}^\star- \bG^\star\|_\infty\leq\frac{2\left\|\bm{\Delta} \bG^\star\right\|_{\infty}}{n-\frac{8d\|\bm{\Delta}\|^2}{n}-1}.
  \end{equation}
Note that
$\rho_F=\frac{1}{10}\min\left\{\sqrt{n},\frac{n}{\|\bm{\Delta}\|}\right\}$ and  $\rho_{\infty}=\frac{1}{4}$ are given by Theorem \ref{hess-positive}. Then we know from the assumption $\|\bm{\Delta}\|\le \frac{n^{3/4}}{80d^{1/2}}$ that $\|\bm{\Delta}\|\leq \min\left\{\frac{n}{80d^{1/2}},\frac{n^{3/4}}{10d^{1/4}}\right\}$, which implies that
\[
\frac{4\sqrt{d}\|\bm{\Delta}\|}{\sqrt{n}}\leq \frac{1}{20}\min\left\{\sqrt{n},\frac{n}{\|\bm{\Delta}\|}\right\}
\]
and consequently $\operatorname{d}_F([\widehat{\bG}],[\bG^\star])\leq \frac{1}{2}\rho_F$. 
In the meantime, from $\|\bm{\Delta}\|\le \frac{n^{3/4}}{80d^{1/2}}$ and $\|\bm{\Delta} \bG^{\star}\|_{\infty}\le \frac{n}{40}$ with $n\ge2$, one has that
\[
\frac{d}{n}\cdot\|\bm{\Delta}\|^2+2\left\|\bm{\Delta} \bG^\star\right\|_{\infty}\leq \frac{1}{8}(n-1),
\]
and then 
\[
\frac{2\left\|\bm{\Delta} \bG^\star\right\|_{\infty}}{n-\frac{8d\|\bm{\Delta}\|^2}{n}-1}\leq \frac{1}{8}.
\]
Thus, we have that
$\|\widehat{\bG}\widehat{\bQ}^\star- \bG^\star\|_\infty\leq \frac{1}{2}\rho_\infty$. The proof is complete.
\end{proof}

Next, we aim at obtaining the Riemannian local error bound property from the conclusion of Theorem \ref{hess-positive}, i.e., the function $f$ is geodesically strongly concave on the region $[\mathcal{N}]:=\{[\bG]\in \mathcal{Q} \mid \bG\in\mathcal{N}\}$ where
\[
\mathcal{N}:=\{\bG\in\odn \mid \operatorname{d}_F([\bG], [\widehat{\bG}])\leq \rho_F\ \text{and}\ \operatorname{d}_{\infty}([\bG], [\widehat{\bG}])\leq \rho_{\infty}\}.
\]
Without loss of generality, we assume the sets $\mathcal{N}$ and $[\mathcal{N}]$ are geodesically convex since we can alway scale the region by a constant letting the set contained by a geodesically convex set. It is more convenient for us to assume the geodesical convexity to promise the existence of a geodesic between any two points.

Now, 
we derive the following corollary showing that (Riemannian) local error bound property holds on not only the quotient manifold $\mathcal{Q}$ and but also the original manifold $\odn$ around the global maximizer $\widehat{\bG}$.


\begin{cor}[(Riemannian) Local Error Bound]\label{cor-EB}
Suppose that $\|\bm{\Delta}\|\le  \frac{n^{3/4}}{20d^{1/2}}$ and $\|\bm{\Delta} \bG^{\star}\|_{\infty}\le \frac{n}{20}$. Then for any $\bG\in\odn$ satisfying 
\[
\operatorname{d}_F([\bG], [\widehat{\bG}])\leq \rho_F\quad \text{and}\quad \|\bG\widehat{\bQ}-\widehat{\bG}\|_{\infty}\leq \rho_{\infty}
\]
with $\rho_F$, $\rho_\infty$ given by Theorem \ref{hess-positive}, it follows that 
\begin{equation}\label{errorbound}
\operatorname{d}_F([\bG],[\widehat{\bG}])\leq\operatorname{d}^{\mathcal{Q}}([\bG],[\widehat{\bG}])\le \frac{10}{n}\|\grad f([\bG])\|_{[\bG]}\le  \frac{10}{n}\|\operatorname{grad} \bar{f}(\bG)\|_F.
\end{equation}
\end{cor}
\begin{proof}
By Theorem~\ref{hess-positive} we know that the function $f$ is geodesically strongly concave on the region
 $[\mathcal{N}]$.
Then from the properties of differentiable geodesically strongly concave functions, there exists an absolute constant $c>0$ such that for any $[\bG]\in[\mathcal{N}]$, one has that
\[
-f([\widehat{\bG}]) \geq -f([\bG])-\left\langle\grad f([\bG]), \dot{\gamma}(0)\right\rangle_{[\bG]}+\frac{n}{10}\|\dot{\gamma}(0)\|_{[\bG]}^2,
\]
where $\gamma:[0,1] \rightarrow \mathcal{Q}$ is any geodesic segment connecting $[\widehat{\bG}]$ and $[\bG]$ contained in $[\mathcal{N}]$.
Thus,
\begin{align*}
\|\grad f([\bG])\|_{[\bG]}\|\dot{\gamma}(0)\|_{[\bG]}
&\ge \left\langle\grad f([\bG]), \dot{\gamma}(0)\right\rangle_{[\bG]}\\
&\geq f([\widehat{\bG}])-f([\bG])+\frac{n}{10}\|\dot{\gamma}(0)\|_{[\bG]}^2\\
&\ge \frac{n}{10}\|\dot{\gamma}(0)\|_{[\bG]}^2,
\end{align*}
and consequently $\|\dot{\gamma}(0)\|_{[\bG]}\leq  \frac{10}{n}\|\grad f([\bG])\|_{[\bG]}$. This together with the fact that $\operatorname{d}_F([\bG],[\widehat{\bG}])\leq\operatorname{d}^{\mathcal{Q}}([\bG],[\widehat{\bG}])\leq \|\dot{\gamma}(0)\|_{[\bG]}$ implies that
\begin{align*}
\operatorname{d}_F([\bG],[\widehat{\bG}])\leq\operatorname{d}^{\mathcal{Q}}([\bG],[\widehat{\bG}])
\leq  \frac{10}{n}\|\grad f([\bG])\|_{[\bG]}
&=\frac{10}{n}\|\overline{\operatorname{grad} f([\bG])}_{\bG}\|_F\notag\\
&=\frac{10}{n}\|\operatorname{Proj}_{\mathcal{H}_{\bG}}(\operatorname{grad} \bar{f}(\bG))\|_F\notag\\
&= \frac{10}{n}\|\operatorname{grad} \bar{f}(\bG)\|_F.
\end{align*}
The proof is complete.
\end{proof}


\begin{remark}
The tolerance of noise level for the local error bound result of the rotation group $\sod$ synchronization problem is improved from $\|\bm{\Delta}\|\lesssim \sqrt{n/d}$ \cite[Proposition 4.5, Theorem 4.3]{zhu2021orthogonal} (which only allow constant noise level when specialized to Gaussian noise) to $\|\bm{\Delta}\|\lesssim n^{3/4}/\sqrt{d}$ in Corollary \ref{cor-EB} deterministically. Near-optimal bound $\|\bm{\Delta}\|\lesssim n/\sqrt{d\log n}$ can be obtain by leave-one-out technique under statistical models, e.g., Gaussian noise setting \cite{zhong2018near,ling2022improved}.
\end{remark}
\begin{remark}
From \cite[Theorem 4.3]{zhu2021orthogonal} we know that when 
$\operatorname{d}_F([\bG], [\bG^{\star}])=\mathcal{O}(\sqrt{n})$ (i.e., $\operatorname{d}_F([\bG], [\widehat{\bG}])=\mathcal{O}(\sqrt{n})$) is sufficient to guarantee the error bound with the residual function related to the fixed point (FP) of the GPM  \cite[Remark 4.4]{zhu2021orthogonal}. However, in Theorem \ref{hess-positive} and Corollary \ref{cor-EB}, besides the tighter requirement $\operatorname{d}_F([\bG], [\widehat{\bG}])=\mathcal{O}(\min\{\sqrt{n},n/\|\bm{\Delta}\|\})$, we need extra assumption that $\operatorname{d}_{\infty}([\bG], [\widehat{\bG}])=\mathcal{O}(1)$
to promise the local error bound \eqref{errorbound} with the characterization of first-order critical point (FOCP) as a residual function. This is consistent with the landscape of orthogonal group sychronization problem that a FP is always a FOCP (cf. \cite[Section 3]{zhu2021orthogonal}) which implies that the quantity of FOCPs are larger than the FPs. Also, we give the following example to show that such additional assumption $\operatorname{d}_{\infty}([\bG], [\widehat{\bG}])=\mathcal{O}(1)$ is necessary.

\begin{example}
Let $d=2$ and $\bm{\Delta}=\bm{0}$, and then $\bG^{\star}=\widehat{\bG}$. Take $\bG\in\odn$  for $i\in[n]$ satisfying that
\[ 
  \bG_{i}=\left\{
  \begin{array}{c@{\quad}l}
  \displaystyle -\widehat{\bG}_{i}, & \textrm{if } i=1, \\
  \noalign{\smallskip}
  \widehat{\bG}_{i}, & \textrm{otherwise}.
  \end{array}
  \right.
\]
Then we know that $\operatorname{grad} \bar{f}(\bG)=S(\bG)\bG=\bm{0}$ since 
\[
S(\bG)=(n-2)\cdot \operatorname{Diag} \left( \left[-\bm{I}_2;\bm{I}_2;\bm{I}_2;\ldots;\bm{I}_2 \right] \right)-\widehat{\bG} \widehat{\bG}^\top.
\]
The global optimum is unique up to rotation while $\operatorname{d}_F([\bG], [\widehat{\bG}])=\sqrt{2}$, which means that $\bG$ is only a FOCP.
\end{example}

\begin{figure}[h]
    \centering
    \includegraphics[width=0.5\textwidth]{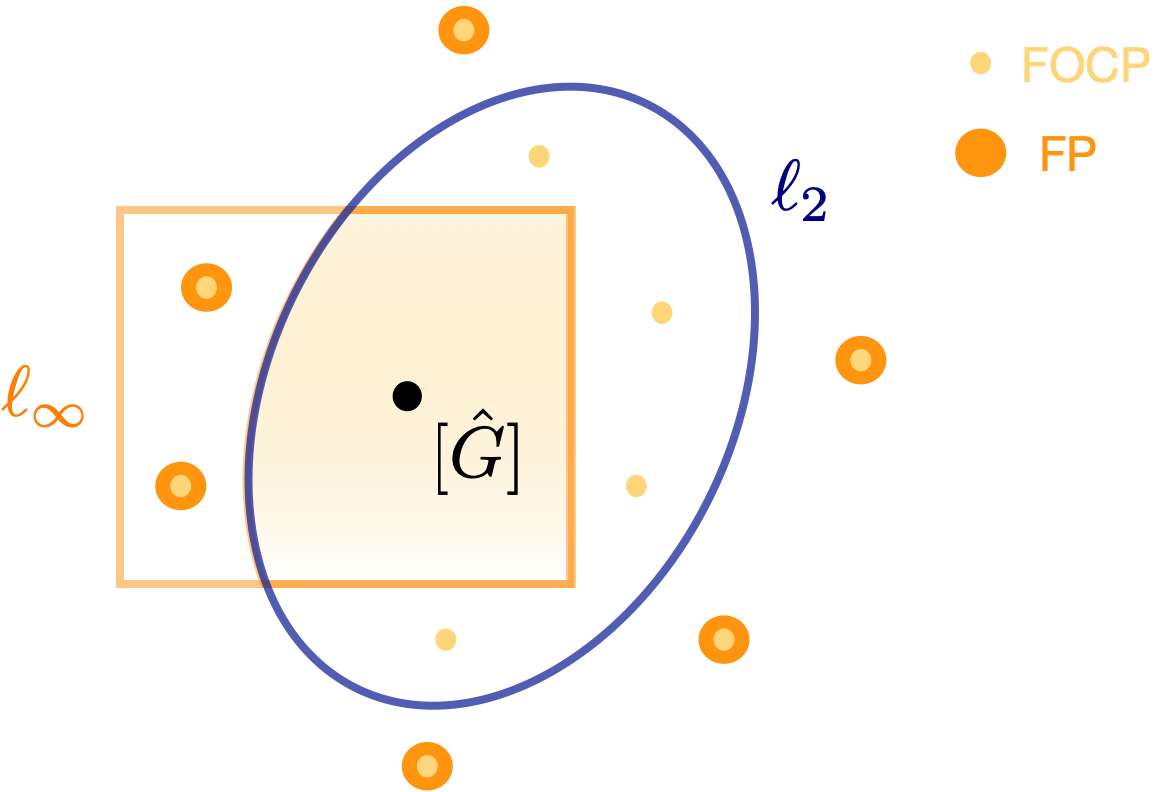}
    \caption{Relationship among geodesically strongly concave region, FPs and FOCPs.
    }
    \label{fig:FPFOCP}
\end{figure}

\end{remark}

\section{Riemannian Gradient Method with Convergence Analysis}
\label{section-rie-algos}

In this section, we will investigate the (quotient) Riemannian  gradient method 
with its convergence analysis for solving \eqref{Problem0}/\eqref{Problem}.
To begin with, we propose the following (quotient) Riemannian gradient method on $\odn$.
\begin{algorithm}[H] 
  \caption{(Quotient) Riemannian Gradient Method}
  \begin{algorithmic}[1]\label{alg: rgd}
    \STATE {\bf Input}: The matrix $\bC$, 
    the stepsize $t_k\ge 0$ and initial point $\bG^0\in\ogn$.\\
    \FOR{$k=0,1,\ldots$}
    \STATE  $\bG^{k+1}:=\overline{\Exp}_{\bG^k}(t_k \operatorname{grad} \bar{f}(\bG^k))$
    \ENDFOR
  \end{algorithmic}
\end{algorithm}
From Lemma \ref{grad-cal-lemma} and Remark \ref{grad-equal} we know that for any $[\bG] \in \mathcal{Q}$ and $\bG \in \pi^{-1}([\bG])$, the unique horizontal lift of the Riemannian gradient $\grad f([\bG])$ of $f$ at $\bG \in \odn$ is given by
\[
\overline{\grad f([\bG])}_{\bG}
=\grad \bar{f}(\bG),
\]
and the exponential map $\Exp$ has the form
\[
\Exp_{[\bG]}(\bm{\xi}_{[\bG]}):=\pi (\overline{\Exp}_{\bG}(\bar{\bm{\xi}}_{\bG})),\quad \text{for each}\ \bm{\xi}_{[\bG]} \in \operatorname{T}_{[\bG]} \mathcal{Q},
\]
where $\bar{\bm{\xi}}_{\bar{\bG}} \in \mathcal{H}_{\bar{\bG}}$ is the horizontal lift of a $\bm{\xi}_{[\bG]} \in \operatorname{T}_{[\bG]} \mathcal{Q}$ at $\bar{\bG}$. Thus, the above Algorithm \ref{alg: rgd} can also be viewed as quotient Riemannian gradient method on the quotient manifold $\mathcal{Q}$.

Now, we are going to prove the sufficient ascent and cost-to-go estimation properties which play an important role to obtain the convergence result of Algorithm \ref{alg: rgd}.
\begin{prop} \label{pro: 3conditioner}
  Let $\alpha\in(0,1)$. Suppose that 
  \begin{equation}\label{step-req}
  0<t_k \leq \frac{1-\alpha}{n(d+1)+\|\bm{\Delta}\|+\sqrt{d}\cdot\|\bm{\Delta} \bG^k\|_{\infty}}.
  \end{equation}
  Then the sequence $\{\bG^{k}\}_{k \geq 0}$  generated by Algorithm \ref{alg: rgd} 
  satisfies:
  \begin{itemize}
  \item[{\rm (i)}] {\rm (Sufficient Ascent)} $f([\bG^{k+1}])-f([\bG^{k}]) \geq \alpha t_k \|\operatorname{grad} \bar{f}(\bG^k)\|_{F}^{2}$;
  \item[{\rm (ii)}] {\rm (Cost-to-Go Estimate)} $f([\widehat{\bG}])-f([\bG^{k}]) \le  (2n+\| \bm{\Delta} \|+ \|  \bm{\Delta}\widehat{\bG} \|_\infty) \cdot \operatorname{d}_F([\bG^{k}], [\widehat{\bG}])^{2}$.
  \end{itemize}
  \end{prop}
  \begin{proof}
    Let $\bG^{k}_{+}:=\bG^{k}+t_k \operatorname{grad} \bar{f}(\bG^k)$. By the definition of $\tilde{f}$ we know that for any $\bm{X}, \bm{Y}\in\mathbb{R}^{nd\times d}$ 
\[
\|\nabla \tilde{f}(\bm{X})-\nabla\tilde{f}(\bm{Y})\|_F=2\|\bC(\bm{X}-\bm{Y})\|_F\leq 2\|\bC\|\cdot\|\bm{X}-\bm{Y}\|_F,
\]
where $\|\bC\|=\|\bG^\star \bG^{\star\top}+\bm{\Delta}\|\leq n+\|\bm{\Delta}\|$. Thus, $\nabla \tilde{f}$ is
$L$-Lipschitz continuous with $L:=2n+2\|\bm{\Delta}\|$, and consequently 
    \begin{align}\label{descent-ineq}
    &f([\bG^{k}])-f([\bG^{k+1}])\notag\\
    =&\ \tilde{f}(\bG^{k})-\tilde{f}\left(\overline{\Exp}_{\bG^{k}}(t_k \operatorname{grad} \bar{f}(\bG^k))\right)\notag\\
   \leq &\ \left\langle\nabla \tilde{f}(\bG^{k}), \bG^{k}-\overline{\Exp}_{\bG^{k}}(t_k \operatorname{grad} \bar{f}(\bG^k))\right\rangle+\frac{L}{2}\left\|\overline{\Exp}_{\bG^{k}}(t_k \operatorname{grad} \bar{f}(\bG^k))-\bG^{k}\right\|_{F}^{2} \notag\\
    = &\ \left\langle\nabla \tilde{f}(\bG^{k}), \bG^{k}-\bG^{k}_{+}+\bG^{k}_{+}-\overline{\Exp}_{\bG^{k}}(t_k \operatorname{grad} \bar{f}(\bG^k))\right\rangle+\frac{L}{2}\left\|\overline{\Exp}_{\bG^{k}}(t_k \operatorname{grad} \bar{f}(\bG^k))-\bG^{k}\right\|_{F}^{2} \notag\\
    \leq &\ -t_k\left\langle\nabla \tilde{f}(\bG^{k}), \operatorname{grad} \bar{f}(\bG^k)\right\rangle+\sqrt{d}\cdot\|\nabla \tilde{f}(\bG^{k})\|_{\infty}\left\|\overline{\Exp}_{\bG^{k}}(t_k \operatorname{grad} \bar{f}(\bG^k))-\bG^{k}_{+}\right\|_F\notag\\
    &+\frac{L}{2}\left\|\overline{\Exp}_{\bG^{k}}(t_k \operatorname{grad} \bar{f}(\bG^k))-\bG^{k}\right\|_{F}^{2}\notag\\
    \leq &\ -t_k\| \operatorname{grad} \bar{f}(\bG^k)\|_F^2+\frac{\sqrt{d}}{2}\|\nabla \tilde{f}(\bG^{k})\|_{\infty}\|t_k \operatorname{grad} \bar{f}(\bG^k)\|_{F}^{2}+\frac{L}{2}\|t_k \operatorname{grad} \bar{f}(\bG^k)\|_{F}^{2},
    \end{align}
    where the last inequality follows from \eqref{retr-key-ineq-1} and \eqref{retr-key-ineq-2}.
    From the definition of $\nabla \tilde{f}$ it follows that 
    \[\|\nabla \tilde{f}(\bG)\|_{\infty} =2\|\bC\bG\|_{\infty}\leq 2\|\bG^{\star\top} \bG\|_F+2\|\bm{\Delta}\bG\|_{\infty} \leq 2n\sqrt{d}+2\|\bm{\Delta} \bG\|_{\infty}\quad \text{for each}\ \bG \in \odn.
    \] 
    This together with \eqref{descent-ineq} implies that
    \begin{align*}
    &f([\bG^{k}])-f([\bG^{k+1}]) \notag\\
    \leq &\ -t_k\| \operatorname{grad} \bar{f}(\bG^k)\|_F^2+\left(nd+\sqrt{d}\cdot\|\bm{\Delta} \bG^k\|_{\infty}+\frac{L}{2}\right)\|t_k \operatorname{grad} \bar{f}(\bG^k)\|_{F}^{2}\notag\\
    = &\ \left(\left(n(d+1)+\|\bm{\Delta}\|+\sqrt{d}\cdot\|\bm{\Delta} \bG^k\|_{\infty}  \right)t_k-1\right)t_k \|\operatorname{grad} \bar{f}(\bG^k)\|_{F}^{2}.
    \end{align*}
   Then we conclude 
    that for any $0<t_k \leq \frac{1-\alpha}{n(d+1)+\|\bm{\Delta}\|+\sqrt{d}\cdot\|\bm{\Delta} \bG^k\|_{\infty}}$,
    \[
    f([\bG^{k}])-f([\bG^{k+1}])\leq -\alpha t_k\|\operatorname{grad} \bar{f}(\bG^k)\|_{F}^{2},
    \]
   which completes
   proof of the conclusion (i).
The analysis for (ii) is similar to the \cite[Proposition 5.4(b)]{zhu2021orthogonal}. For the sake of completeness, we present the proof here. Assume that for all $k \geq 0$, $\widehat{\bQ}^k\in \mathcal{O}(d)$ satisfies $\operatorname{d}_F([\bG^k],[\widehat{\bG}])=\Vert \bG^k \widehat{\bQ}^k-\widehat{\bG} \Vert_F$. 
Let $D:\odn \rightarrow \mathbb{S}^{nd}$ be defined as
  \begin{equation} \label{eq: defineDG}
        D(\bG):=\operatorname{Diag} \left( \left[\bm{U}_{\bC_{1,:} \bG}\bm{\Sigma}_{\bC_{1,:} \bG}\bm{U}_{\bC_{1,:} \bG}^\top;\ldots;\bm{U}_{\bC_{n,:} \bG}\bm{\Sigma}_{\bC_{n,:} \bG}\bm{U}_{\bC_{n,:} \bG}^\top \right] \right)-\bC.
  \end{equation} 
From \cite[Lemma 3.6, Lemma 3.3(b)]{zhu2021orthogonal} 
we know that 
$D(\widehat{\bG})\widehat{\bG}=\bm{0}$. This together with
\cite[Lemma 3.3(d)]{zhu2021orthogonal} shows that
  \begin{align*}
  f([\widehat{\bG}])-f([\bG^{k}])
  &=\operatorname{tr}(\widehat{\bG}^{\top} \bC \widehat{\bG})-\operatorname{tr}(\bG^{k\top} \bC \bG^k)=\sum\limits_{i=1}^n \Vert {\bC_{i,:}} \widehat{\bG}\Vert_*-\operatorname{tr}(\bG^{k\top} \bC \bG^k)\notag\\
  &=\operatorname{tr}\left(\bG^{k\top} \left[\operatorname{Diag}\left(\left[\bm{U}_{\bC_{1,:} \widehat{\bG}}\bm{\Sigma}_{\bC_{1,:} \widehat{\bG}}\bm{U}_{\bC_{1,:} \widehat{\bG}}^\top;\dots;\bm{U}_{\bC_{n,:} \widehat{\bG}}\bm{\Sigma}_{\bC_{n,:} \widehat{\bG}}\bm{U}_{\bC_{n,:} \widehat{\bG}}^\top\right]\right)-\bC\right] \bG^{k} \right)\notag\\
  &=\operatorname{tr}\left((\bG^{k}\widehat{\bQ}^k-\widehat{\bG})^{\top} D(\widehat{\bG})(\bG^{k}\widehat{\bQ}^k-\widehat{\bG})\right) \notag\\
  &\leq\left(\|\bC\|+\max_{i\in[n]}\|\bC_{i,:} \widehat{\bG}\|\right)\cdot \operatorname{d}_F([\bG^{k}], [\widehat{\bG}])^{2},
  \end{align*}
where 
\begin{align*}
   & \| \bC \| +  \max_{i\in[n]}\|\bC_{i,:} \widehat{\bG}\|  \leq n + \| \bm{\Delta} \|  + \max_{i\in[n]}\| \bG^\star_i \bG^{\star \top}\widehat{\bG} \|  + \|  \bm{\Delta}\widehat{\bG} \|_\infty   
      \leq 2n+\| \bm{\Delta} \|+ \|  \bm{\Delta}\widehat{\bG} \|_\infty.
  \end{align*} 
The proof is complete.
\end{proof}

Now we are ready to state the convergence theorem of Algorithm \ref{alg: rgd} as the main result for this section.

\begin{theorem}[Convergence Theorem] \label{thm:convergence-rate}
Suppose that $\|\bm{\Delta}\|\le \frac{n^{3/4}}{20d^{1/2}}$, $\|\bm{\Delta} \bG^{\star}\|_{\infty}\le \frac{n}{20}$.
  Suppose further that the sequence $\{ \bG^k\}_{k\geq 0}$ generated by Algorithm \ref{alg: rgd} satisfies that for $k\ge0$,
  \begin{equation*}
  0<\underline{t}\leq t_k \leq \frac{1-\alpha}{n(d+1)+\|\bm{\Delta}\|+\sqrt{d}\cdot\|\bm{\Delta} \bG^k\|_{\infty}}
  \end{equation*}
with $\alpha\in(0,1)$ for some $\underline{t}>0$, and
\begin{equation}\label{stay-in-ball-requ}
\operatorname{d}_F([\bG^k], [\widehat{\bG}])\leq 
\rho_F,
 \quad \|\bG^k\bQ^k-\widehat{\bG}\|_\infty \leq \rho_\infty,
\end{equation}
where $\bQ^k\in \mathcal{O}(d)$ satisfies $\operatorname{d}_F([\bG^k],[\widehat{\bG}])=\Vert \bG^k \bQ^k-\widehat{\bG} \Vert_F$ and $\rho_F$, $\rho_{\infty}$ are given by Theorem \ref{hess-positive}.
  Then the sequence $\{f([\bG^k])\}_{k\ge0}$ converges Q-linearly to $f([\widehat{\bG}])$, i.e., 
  \begin{align*}
    f([\widehat{\bG}]) - f([\bG^k]) \leq \left(f([\widehat{\bG}]) - f([\bG^0]) \right) \lambda^k, \quad  \text{for each $k\ge0$},
  \end{align*}
  and the sequence $\{\bG^k\}_{k\geq 0}$ converges R-linearly to some $\bG^*\in[\widehat{\bG}]$, i.e.,
  \begin{align*}
    \operatorname{d}_F([\bG^k],[\widehat{\bG}])\leq
\|\bG^{k}-\bG^*\|_F \leq a \left(f([\widehat{\bG}]) - f([\bG^0])\right)^{\frac{1}{2}} \lambda^{\frac{k}{2}}, \quad  \text{for each $k\ge0$}, 
  \end{align*}
  where 
$a > 0$, $\lambda \in (0,1)$ are constants that depend only on $n$ and $d$.
\end{theorem}

\begin{proof}
Since all assumptions are satisfied in
  Corollary \ref{cor-EB},
   we know for all $k\ge0$ that 
  \begin{equation}\label{EB-use-key}
  \operatorname{d}_F([\bG^k],[\widehat{\bG}])\leq  \frac{10}{n}\|\grad f([\bG^k])\|_{[\bG^k]}\le  \frac{10}{n}\|\operatorname{grad} \bar{f}(\bG^k)\|_F.
  \end{equation}
    Then combining \eqref{EB-use-key} with Proposition \ref{pro: 3conditioner} it follows that
  \begin{equation}\label{linear-key0}
  \begin{aligned}
  f([\widehat{\bG}]) - f([\bG^k])\leq 3 n\cdot \operatorname{d}_F([\bG^k],[\widehat{\bG}])^2
  &\leq \frac{300 }{n}\|\grad f([\bG^k])\|_F^2\\
  &\leq \frac{300 }{\alpha n\underline{t}}\left( f([\bG^{k+1}])  - f([\bG^k]) \right), 
  \end{aligned}
  \end{equation}
  where the first inequality is from $2n+\|\bm{\Delta}\|+\|\bm{\Delta}\widehat{\bG}\|_{\infty}\leq3n$, since
      \[
      \|\bm{\Delta} \widehat{\bG}\|_{\infty} 
      \leq \| \bm{\Delta}\|\cdot \operatorname{d}_F([\widehat{\bG}], [\bG^{\star}])+\left\|\bm{\Delta} \bG^{\star}\right\|_{\infty}
      \leq \frac{4\sqrt{d}\|\bm{\Delta}\|^2}{\sqrt{n}}+\left\|\bm{\Delta} \bG^{\star}\right\|_{\infty}
      \leq\frac{n}{10}.
    \]
  Then we have that 
  \begin{equation}\label{linear-key1}
  \begin{aligned}
  f([\widehat{\bG}]) - f([\bG^{k+1}]) 
  &= f([\widehat{\bG}]) - f([\bG^k])  - \left( f([\bG^{k+1}]) - f([\bG^k]) \right) \\
  &\leq \left( \frac{300 }{\alpha n\underline{t}} - 1 \right) \left( f([\bG^{k+1}]) - f([\widehat{\bG}]) + f([\widehat{\bG}]) - f([\bG^k]) \right).
  \end{aligned}
  \end{equation}
      Since $f([\widehat{\bG}])\ge f([\bG^k])$ for $k\in\mathbb{N}$, we may assume without loss of generality that $a'=\frac{300 }{\alpha n\underline{t}}>1$. Thus, from \eqref{linear-key1} one has that
      \begin{equation}\label{func-linear}
      f([\widehat{\bG}])-f([\bG^{k+1}]) \le \frac{a'-1}{a'}\left( f([\widehat{\bG}]) - f([\bG^k]) \right), 
      \end{equation}
      which yields with $\lambda=\tfrac{a'-1}{a'}\in(0,1)$ that
      \begin{equation}\label{linear-key2}
      f([\widehat{\bG}])-f([\bG^k]) \le \left( f([\widehat{\bG}])-f([\bG^0]) \right) \lambda^k.
      \end{equation}
From \eqref{pro: 3conditioner} and \eqref{retr-key-ineq-1}  we know that for all $k$ that
\[
\frac{\alpha}{\overline{t}}\|\bG^{k+1}-\bG^k\|_F^2\leq \alpha t_k \|\operatorname{grad} \bar{f}(\bG^k)\|_{F}^{2}  \leq f([\bG^{k+1}])-f([\bG^{k}]) \leq f([\widehat{\bG}])-f([\bG^{k}]).
\]
Combining this with \eqref{func-linear}, it follows that for any $k$, $l$ that 
\[
\frac{\alpha}{\overline{t}}\|\bG^{k+l+1}-\bG^{k+l}\|_F^2
\leq f([\widehat{\bG}])-f([\bG^{k+l}])\leq \left(\frac{a'-1}{a'}\right)^l \left( f([\widehat{\bG}]) - f([\bG^k]) \right).
\]
Thus, one has that
\begin{equation}\label{sequence-converge-key}
\begin{aligned}
\|\bG^{k+l}-\bG^k\|_F
\leq\sum_{i=1}^l \|\bG^{k+i}-\bG^{k+i-1}\|_F
&\leq \sum_{i=1}^l \rho^{i} \left( f([\widehat{\bG}]) - f([\bG^k]) \right)^{\frac{1}{2}} \\
&\leq  \frac{\rho(1-\rho^l)}{1-\rho} \left( f([\widehat{\bG}]) - f([\bG^k]) \right)^{\frac{1}{2}},
\end{aligned}
\end{equation}
where $\rho:=\sqrt{\frac{\overline{t}}{\alpha}}\left(\frac{a'-1}{a'}\right)^{\frac{1}{2}}$. Consequently, $\{\bG^k\}_{k\ge0}$ is a Cauchy sequence. Let $l\rightarrow\infty$ and we know from \eqref{sequence-converge-key} that for some $\bG^*\in[\widehat{\bG}]$ one has that
\begin{align*}
\operatorname{d}_F([\bG^k],[\widehat{\bG}])\leq
\|\bG^{k}-\bG^*\|_F
\leq  \frac{\rho}{1-\rho} \left( f([\widehat{\bG}]) - f([\bG^k]) \right)^{\frac{1}{2}}
\leq a \left( f([\widehat{\bG}])-f([\bG^0]) \right)^{\frac{1}{2}} \lambda^{\frac{k}{2}}
\end{align*}
with $a:=\frac{\rho}{1-\rho}$. This completes the proof. 
  \end{proof}

In order to apply Theorem \ref{thm:convergence-rate} for deriving the global convergence of Algorithm \ref{alg: rgd}, we still need to find a proper initialization strategy and promise the iterates of the algorithm will not leave the effective domain of the Riemannian local error bound once it steps in (i.e., \eqref{stay-in-ball-requ}). 

We first handle the latter part for keeping the iterates in the effective domain of Corollary \ref{cor-EB}. To begin with, we have the following lemma with proof in Appendix \ref{appendix-B}. 
\begin{lemma}\label{lemma-dist-rotation}
Let $\bG, \bm{H}_1, \bm{H}_2 \in \mathbb{R}^{n d \times d}$ satisfy $\bG^{\top} \bG=\bm{H}_1^\top \bm{H}_1=\bm{H}_2^\top \bm{H}_2=n \bm{I}_{d}$ and $\bQ$ be such that $\operatorname{d}_{F}([\bm{H}_1], [\bm{H}_2])=\left\|\bm{H}_1\bQ-\bm{H}_2\right\|_{F}$. 
Suppose that $\varepsilon_{1}, \varepsilon_{2} \in\left(0,\frac{1}{2}\right)$ with 
\[\operatorname{d}_{F}([\bm{H}_1], [\bG])=\|\bm{H}_1-\bG\|_{F} \leq\varepsilon_{1} \sqrt{n}\quad \text{and}\quad \operatorname{d}_{F}([\bm{H}_2], [\bG])=\|\bm{H}_2-\bG\|_{F} \leq \varepsilon_{2} \sqrt{n}.\]
 Then it follows that
\begin{align*}
\|\bQ-\bm{I}_d\|_F  
&\leq \frac{ 2\min\{\operatorname{d}_{F}([\bm{H}_1], [\bG]),\operatorname{d}_{F}([\bm{H}_2], [\bG])\}}{n-2\max\{\operatorname{d}_{F}([\bm{H}_1], [\bG])^2,\operatorname{d}_{F}([\bm{H}_2], [\bG])^2\}} \cdot \operatorname{d}_{F}([\bm{H}_1], [\bm{H}_2])\notag\\
&\leq \frac{4\min\{\varepsilon_1,\varepsilon_2\}}{\sqrt{n}}\cdot \operatorname{d}_{F}([\bm{H}_1], [\bm{H}_2]).
\end{align*}
\end{lemma}

With the help of Lemma \ref{lemma-dist-rotation}, we have the following result to keep the iterates in the region. As an example, we only show the case when $d=3$, the conclusion and proof of Proposition \ref{prop: stayintheball} is similar when $d\neq3$ depending on the slight difference of explicit form of the exponential map. 

\begin{prop}[Staying in Ball for $d=3$]\label{prop: stayintheball} 
Let $\bG\in\mathcal{O}(3)^{n}$ and $\bQ^\star:=\operatorname{argmin}_{\bQ} \|\bG\bQ- \bG^\star\|_F$. Let $\bm{E}^{\star}:=[\bm{E}_1^{\star};\ldots;\bm{E}_n^{\star}]\in\operatorname{Skew}(3)^n$ satisfy $\bG_{i}^{\star}=\bG_{i} \exp(\bm{E}_{i}^{\star})\bQ^\star$ for $i\in[n]$. Suppose that the following assumptions hold: 
\begin{itemize}
\item {\rm (Noise)} $\|\bm{\Delta}\|\le \frac{n}{50}$ and $\|\bm{\Delta} \bG^{\star}\|_{\infty}\le \frac{n}{400}$;
\item {\rm (Distance)} 
$\|\bm{E}^\star\|_F\leq \frac{1}{200}\cdot\min\left\{\sqrt{n},\frac{n}{\|\bm{\Delta}\|}\right\}$ and $\|\bm{E}^\star\|_\infty \leq \frac{1}{10}$; 
\item {\rm (Stepsize)} $t \le \frac{1}{2n}$.
\end{itemize}
Then  it follows that
\[
\operatorname{d}_F([\bG^+], [\bG^\star])\leq \frac{1}{200}\cdot\min\left\{\sqrt{n},\frac{n}{\|\bm{\Delta}\|}\right\} \quad \text{and} \quad \|\bG^+\bQ_+^\star-\bG^\star\|_\infty \leq \frac{1}{10},
\]
where $\bG^{+}:=\bG\exp(t\grad \bar{f}(\bG))$ and $\bQ_+^\star:=\underset{\bQ}{\operatorname{argmin}}\ \|\bG^+\bQ- \bG^\star\|_F$.
\end{prop}
\begin{proof} 
First, from Lemma \ref{lemma-exp-lip} we have the following relationship that 
\begin{equation}\label{Euc-Rie-dist-infty}
\|\bG\bQ^\star-\bG^\star\|_\infty=\max_{i\in[n]}\|\bG_i \bQ^\star- \bG_i^\star\|_F=\max_{i\in[n]}\|\exp(\bm{E}_{i}^\star)-\bm{I}_d\|_F\leq\|\bm{E}^\star\|_{\infty}
\end{equation}
and
\begin{equation}\label{Euc-Rie-dist-F}
\operatorname{d}_{F}([\bG],[\bG^\star])=\|\bG\bQ^\star- \bG^\star\|_F=\sqrt{\sum_{i=1}^n\|\exp(\bm{E}_{i}^\star)-\bm{I}_d\|_F^2}\leq\|\bm{E}^\star\|_F.
\end{equation}
Also, we have the following useful bound
\begin{align} 
  \| \bm{\Delta} \bG \|_\infty \leq \| \bm{\Delta}(\bG \bQ^\star - \bG^\star)\|_\infty +   \| \bm{\Delta} \bG^\star \|_\infty 
  &\leq \| \bm{\Delta} \bG^\star \|_\infty + \| \bm{\Delta} \|\cdot \operatorname{d}_F([\bG],[\bG^\star])\notag
\end{align} 
which implies that
$\|\bm{\Delta} \bG\|_{\infty} \leq \kappa n$ with $\kappa=\frac{1}{100}$.

Since $\operatorname{d}_F([\bG],[\bG^\star])^2 
 =2(nd-\Vert \bG^{\star\top} \bG\Vert_*)$, there exists $\bm{U}\in \od$ such that $\bG^{\star\top} \bG \bQ^\star= \bm{U}\bm{\Sigma}_{\bG^{\star\top} \bG} \bm{U}^\top$. Then from the definition of $\grad \bar{f}(\bG)$, for each $i\in[n]$ one has that
\begin{equation}\label{dist-key-grad}
\begin{aligned}
& (\grad \bar{f}(\bG))_{i}\\
= &\ 2(\bC-\operatorname{symblockdiag}(\bC \bG \bG^\top )) \bG \\ 
= &\ \bC_{i,:} \bG-\bG_i \bG^\top \bC_{i,:}^\top \bG_i\\
= &\ \left(\bG_i^\star \bG^{\star\top}\bG \bG_i^\top-\bG_i \bG^\top \bG^{\star} \bG_i^{\star\top}+\bm{\Delta}_{i,:} \bG \bG_i^\top -\bG_i \bG^\top \bm{\Delta}_{i,:}\right)\bG_i\\
= &\ \left(\bG_i^\star  \bm{U}\bm{\Sigma}_{\bG^{\star\top} \bG} \bm{U}^\top \bQ^{\star\top} \bG_i^\top-\bG_i \bQ^\star \bm{U}\bm{\Sigma}_{\bG^{\star\top} \bG} \bm{U}^\top  \bG_i^{\star\top}+\bm{\Delta}_{i,:} \bG \bG_i^\top -\bG_i \bG^\top \bm{\Delta}_{i,:}\right)\bG_i\\
= &\ \left(\bG_i\exp(\bm{E}_{i}^\star) \bar{\bm{U}}\bm{\Sigma}_{\bG^{\star\top} \bG} \bar{\bm{U}}^\top  \bG_i^\top-\bG_i \bar{\bm{U}}\bm{\Sigma}_{\bG^{\star\top} \bG} \bar{\bm{U}}^\top \exp(\bm{E}_{i}^\star)^\top \bG_i^\top+\bm{\Delta}_{i,:} \bG \bG_i^\top -\bG_i \bG^\top \bm{\Delta}_{i,:}\right)\bG_i,
\end{aligned}
\end{equation}
where $\bar{\bm{U}}:=\bQ^\star \bm{U}$. On the other hand, we know the explicit form of the exponential map for each $\bm{H}\in{\rm Skew}(3)$: 
\begin{equation}\label{exp-d3}
\exp(\bm{H})=\bm{I}_3+\frac{\sqrt{2}\sin\left(\frac{\sqrt{2}}{2}\|\bm{H}\|_F\right)}{\|\bm{H}\|_F}\bm{H}+\frac{2\left(1-\cos\left(\frac{\sqrt{2}}{2}\|\bm{H}\|_F\right)\right)}{\|\bm{H}\|_F^2}\bm{H}^2.
\end{equation}   
Then from \eqref{Euc-Rie-dist-F}, \eqref{dist-key-grad} and \eqref{exp-d3} we know that 
\begin{align}\label{grad-minus-Estar}
& \|t\bG_i^\top\grad \bar{f}(\bG)_i-\bm{E}_i^\star\|_F   \notag\\
= &\ \left\|t\left(\exp(\bm{E}_{i}^\star) \bar{\bm{U}}\bm{\Sigma}_{\bG^{\star\top} \bG} \bar{\bm{U}}^\top  - \bar{\bm{U}}\bm{\Sigma}_{\bG^{\star\top} \bG} \bar{\bm{U}}^\top \exp(\bm{E}_{i}^\star)^\top +\bG_i^\top\bm{\Delta}_{i,:} \bG  - \bG^\top \bm{\Delta}_i\bG_i\right)-\bm{E}_i^\star\right\|_F\notag\\
\leq &\  \left\|t\sin \left(\frac{\sqrt{2}}{2}\|\bm{E}_i^\star\|_F\right)\frac{\sqrt{2}}{\|\bm{E}_i^\star\|_F}\left( \bm{E}_i^\star \bar{\bm{U}}\bm{\Sigma}_{\bG^{\star\top} \bG} \bar{\bm{U}}^\top  + \bar{\bm{U}}\bm{\Sigma}_{\bG^{\star\top} \bG} \bar{\bm{U}}^\top \bm{E}_i^\star \right)-\bm{E}_i^\star\right\|_F \notag\\
&+\left\|\frac{2t\left(1-\cos\left(\frac{\sqrt{2}}{2}\|\bm{E}_i^\star\|_F\right)\right)}{\|\bm{E}_i^\star\|_F^2}(\bm{E}_i^{\star2} \bar{\bm{U}}\bm{\Sigma}_{\bG^{\star\top} \bG} \bar{\bm{U}}^\top  - \bar{\bm{U}}\bm{\Sigma}_{\bG^{\star\top} \bG} \bar{\bm{U}}^\top \bm{E}_i^{\star2} )\right\|_F+2t\|\bm{\Delta}_{i,:} \bG\|_F\notag\\
\leq &\ \|\bm{E}_i^\star\|_F\| 2t \bm{\Sigma}_{\bG^{\star\top} \bG} -\bm{I}_d\|+2t\|\bm{\Sigma}_{\bG^{\star\top} \bG}\|\cdot\|\bm{E}_i^\star\|_F^2+2t\|\bm{\Delta}_{i,:} \bG\|_F\notag\\
\leq &\ \|\bm{E}_i^\star\|_F (1-2t\sigma_{\min}( \bm{\Sigma}_{\bG^{\star\top} \bG}) )+2tn\|\bm{E}_i^\star\|_F^2+2t\|\bm{\Delta}_{i,:} \bG\|_F\notag\\
\leq &\  \left(1-t\left(2-\frac{1}{n}\operatorname{d}_F([\bG^\star],[\bG])^2\right) n \right)\|\bm{E}_i^\star\|_F+2tn\|\bm{E}_i^\star\|_F^2+2t\|\bm{\Delta}_{i,:} \bG\|_F\notag\\
\leq &\  \left(1-t\left(2-2\|\bm{E}_i^\star\|_F-\frac{1}{n}\|\bm{E}^\star\|_F^2\right) n \right)\|\bm{E}_i^\star\|_F+2t\|\bm{\Delta}_{i,:} \bG\|_F,
\end{align}
where $1-2t\sigma_{\min}( \bm{\Sigma}_{\bG^{\star\top} \bG}) \geq 0$ is from $1-2tn\ge 0$ (by $t\leq\frac{1}{2n}$), and the last two inequalities hold since
\[
\sigma_{\min }(\bG^{\star\top} \bG)\ge n-\sum_{l=1}^{d}\left(n-\sigma_{l}(\bG^{\star\top} \bG)\right)=n-(nd-\Vert \bG^{\star\top} \bG\Vert_*)=n-\frac{1}{2}\operatorname{d}_F([\bG^\star],[\bG])^2
\]
(noting that $\|\bG^{\star\top} \bG\| \leq n$ implies $0 \leq \sigma_{l}(\bG^{\star\top} \bG) \leq n$ for $l\in[d]$).
Since $\bG^{+}=\Exp_{\bG}(t\grad \bar{f}(\bG))$ and 
the 1-Lipschitz property of the exponential map on skew-symmetric space (i.e., $\|\exp(\bm{X})-\exp(\bm{Y})\|_F\leq \|\bm{X}-\bm{Y}\|_F$ for all $\bm{X},\bm{Y}\in \operatorname{Skew}(d)$) from Lemma \ref{lemma-exp-lip}, one has for each $i\in[n]$ that
\begin{align}\label{dist-key-GplusGstar}
\|  \bG_{i}^{+}\bQ^\star-\bG_{i}^{\star}\|_F
&=\|  \bG_{i}\exp(t\bG_i^\top\grad \bar{f}(\bG)_i)\bQ^\star-\bG_{i}^{\star}\|_F\notag\\
&=\|  \exp(t\bG_i^\top\grad \bar{f}(\bG)_i)-\exp(\bm{E}_i^\star)\|_F\notag\\
&\leq\| t\bG_i^\top\grad \bar{f}(\bG)_i-\bm{E}_i^\star\|_F. 
\end{align}
Thus, we deduce from \eqref{dist-key-grad}, \eqref{exp-d3}, \eqref{grad-minus-Estar} and \eqref{dist-key-GplusGstar} with Cauchy-Schwarz inequality $\sum_{i=1}^n(a_i+b_i)^2\leq\left(\sqrt{\sum_{i=1}^n a_i^2}+\sqrt{\sum_{i=1}^n b_i^2}\right)^2$ that
\begin{align*}
\operatorname{d}_{F}([\bG^{+}], [\bG^{\star}])^2
&\leq\sum_{i=1}^n\|  \bG_{i}^{+}\bQ^\star-\bG_{i}^{\star}\|_F^2 \leq \sum_{i=1}^n \|t\bG_i^\top\grad \bar{f}(\bG)_i-\bm{E}_i^\star\|_F^2 \notag\\
&\leq \left[\left(1-t\left(2-2\|\bm{E}^\star\|_\infty-\frac{1}{n}\|\bm{E}^\star\|_F^2\right) n \right)\|\bm{E}^\star\|_F+2t\|\bm{\Delta} \bG\|_\infty\right]^{2}.\notag
\end{align*}
Thus, from 
$\|\bm{E}^\star\|_F\leq \varepsilon\min\left\{\sqrt{n},\frac{n}{\|\bm{\Delta}\|}\right\}$ with $\varepsilon:=\frac{1}{200}$, $\|\bm{E}_i^\star\|_F\leq \frac{1}{10}$, $t\le\frac{1}{2n}$ and $\|\bm{\Delta} \bG\|_{\infty} \leq \kappa n$, it follows 
that
\[
\operatorname{d}_F([\bG^{+}], [\bG^{\star}])^2\leq \varepsilon^2\min\left\{n,\frac{n^2}{\|\bm{\Delta}\|^2}\right\}.
\]
Again by \eqref{dist-key-grad}, we also have that
\begin{align*}
 \|\grad \bar{f}(\bG)_i\|_F   
=\ & \|\exp(\bm{E}_{i}^\star) \bar{\bm{U}}\bm{\Sigma}_{\bG^{\star\top} \bG} \bar{\bm{U}}^\top  - \bar{\bm{U}}\bm{\Sigma}_{\bG^{\star\top} \bG} \bar{\bm{U}}^\top \exp(\bm{E}_{i}^\star)^\top +\bG_i^\top\bm{\Delta}_{i,:} \bG  - \bG^\top \bm{\Delta}_{i,:}\bG_i\|_F\notag\\
\leq\ & 2\|\bm{E}_i^\star\|_F\|  \bm{\Sigma}_{\bG^{\star\top} \bG} \|+2\|  \bm{\Sigma}_{\bG^{\star\top} \bG} \|\cdot\|\bm{E}_i^\star\|_F^2+2\|\bm{\Delta}_{i,:} \bG\|_F\notag\\
\leq\ & 2n(1+\|\bm{E}_i^\star\|_F)\|\bm{E}_i^\star\|_F+2\|\bm{\Delta}_{i,:} \bG\|_F\notag
\end{align*}
and consequently
\begin{align*}
 \|\grad \bar{f}(\bG)\|_F = \sqrt{\sum_{i=1}^n \|\grad \bar{f}(\bG)_i\|_F^2}  
&\leq 2\sqrt{\sum_{i=1}^n\left(n(\|\bm{E}_i^\star\|_F+\|\bm{E}_i^\star\|_F^2)+\|\bm{\Delta}_{i,:} \bG\|_F\right)^2}\notag\\
&\leq 2\sqrt{2\left((1+\|\bm{E}^\star\|_{\infty})^2n^2\|\bm{E}^\star\|_F^2+\|\bm{\Delta} \bG\|_F^2\right)}.
\end{align*}
Then, one has that
\begin{align}\label{Gplus-G-dist}
\operatorname{d}_{F}([\bG^{+}],[\bG])^{2} \leq\|\bG-\bG^{+}\|_{F}^{2} 
\leq\|t\grad \bar{f}(\bG)\|_{F}^{2}
&\leq \frac{t^{2} n^{3}}{1000}.
\end{align}
Note from \eqref{Euc-Rie-dist-F} that $\operatorname{d}_{F}([\bG], [\bG^{\star}])^{2}=\|\bG \bQ^{\star}-\bG^{\star}\|_{F}^{2} \leq \varepsilon^{2} n$. Let $\bQ^\star_{+}:=\operatorname{argmin}_{\bQ}\left\|\bG^{+} \bQ-\bG^{\star}\right\|_{F}$ and $\bQ_+:=\operatorname{argmin}_{\bQ}\left\|\bG-\bG^{+} \bQ\right\|_{F}$.
Then from Lemma \ref{lemma-dist-rotation} with \eqref{Gplus-G-dist} and the fact that 
\[
\bQ_+^\top \bQ_+^{\star}\bQ^{\star\top}=\underset{\bQ}{\operatorname{argmin}}\left\|\bG^{\star}\bQ^{\star\top} -\bG^{+}\bQ_+  \bQ\right\|_{F},
\]
one has with $\varepsilon_1:=\varepsilon$, $\varepsilon_2:=\frac{tn}{\sqrt{1000}}$ and $\operatorname{d}_F([\bG^{+}], [\bG^{\star}])\leq \varepsilon\sqrt{n}$ that
\[
\|\bQ_+^\top \bQ_+^{\star}\bQ^{\star\top}-\bm{I}_d\|_F \leq 4 \min \left\{\varepsilon_{1}, \varepsilon_{2}\right\}\cdot \frac{\operatorname{d}_{F}([\bG^+],[\bG^\star])}{\sqrt{n}}\leq 4\varepsilon\min\left\{\varepsilon,\frac{tn}{\sqrt{1000}}\right\}.
\]
On the other hand we know that
\begin{align*}
\|\bQ_+-\bm{I}_d\|_F
&= \left\|\bQ_+-\frac{1}{n}\bG^{+\top}\bG+\frac{1}{n}\bG^{+\top}\bG-\bm{I}_d\right\|_F\leq \frac{2}{n}\|\bG^{+\top}\bG-n\bm{I}_d\|_F
\leq \frac{nt}{15},
\end{align*}
where the last inequality is from \eqref{Gplus-G-dist} and
\begin{align*}
\|\bG^{+\top}\bG-n\bm{I}_d\|_F
&\leq \sum_{i=1}^n\left\|\exp(t\grad \bar{f}(\bG)_i)-\bm{I}_d\right\|_F\leq \sum_{i=1}^n\left\|t\grad \bar{f}(\bG)_i\right\|_F\leq t\sqrt{n}\left\|\grad \bar{f}(\bG)\right\|_F.
\end{align*}
Hence, we obtain
\begin{align}\label{dist-key-littlegplusgstar}
\|\bQ_+^\star-\bQ^\star\|_F
&\leq \|\bQ_+^\top \bQ_+^{\star}\bQ^{\star\top}-\bm{I}_d\|_F+\|\bQ_+-\bm{I}_d\|_F\leq 4\varepsilon\min\left\{\varepsilon,\frac{tn}{\sqrt{1000}}\right\}+\frac{n t}{15}.
\end{align}
Finally, from \eqref{grad-minus-Estar}, \eqref{dist-key-GplusGstar} and \eqref{dist-key-littlegplusgstar} we conclude that
\begin{align*}
  \| \bG_{i}^{+}\bQ_+^\star-\bG_{i}^{\star}\|_F 
\leq\ &\|\bG_{i}^+\bQ^\star-\bG_{i}^\star\|_F+\|\bQ^\star-\bQ_+^\star\|_F \\
\leq\ &  \left(1-t\left(2-2\|\bm{E}^\star\|_\infty-\frac{1}{n}\|\bm{E}^\star\|_F^2\right) n \right)\|\bm{E}_i^\star\|_F+2t\|\bm{\Delta}_{i,:} \bG\|_F\notag\\
&+4\varepsilon\min\left\{\varepsilon,\frac{tn}{\sqrt{1000}}\right\}+\frac{nt}{15}\notag\\
\leq\ & \left(1-2tn+t\varepsilon^2n\right)\cdot \frac{1}{10}+2\cdot \frac{1}{10}^2tn+2\kappa tn+4\varepsilon\min\left\{\varepsilon,\frac{tn}{\sqrt{1000}}\right\} +\frac{nt}{15} \\
\leq\ & \frac{1}{10}.
\end{align*}
The proof is complete.
\end{proof}

In Proposition \ref{prop: stayintheball}, for any given $\bG$ and $\bG^\star$ in the same connect component of $\odn$, the skew-symmetric matrix $\bm{E}^\star$ satisfying $\bG_{i}^{\star}=\bG_{i} \exp(\bm{E}_{i}^{\star})\bQ^\star$ for all $i\in[n]$ is well-defined since the exponential map $\Exp(\cdot)$ for $\sod$ is surjective  from 
\cite[Corollary 11.10]{hall2015lie}.

Also, we have the following lemma which is important by relating the Riemannian distance and the distance in the ambient Euclidean space; see details in Remark \ref{convergence-key-remark}. It plays a crucial role in making Proposition \ref{prop: stayintheball} effective, as it enables the transfer of distances to the same measurement. The proof of this lemma can be found in Appendix \ref{appendix-C}.


\begin{lemma}\label{lemma-inj-cvx-radius}
For the orthogonal/rotation group $\og$, the injective and convex radius\footnote{The injectivity and convexity
radius related to $\mathcal{M}$ (letting $\mathbf{B}(\bm{X}, r):=\{\bm{Y} \in  \mathcal{M}\mid \operatorname{d}^{\mathcal{M}}(\bm{X}, \bm{Y})<r\}$): 
\begin{align*}
r_{\rm inj}(\mathcal{M})&:=\inf_{\bm{X}\in\mathcal{M}}\sup\left\{r>0:\Exp_{\bm{X}}(\cdot)
\mbox{ is a diffeomorphism on } \mathbf{B}(\bm{0},r)\subset \operatorname{T}_{\bm{X}}\mathcal{M}
\right\},\notag\\
r_{\operatorname{cvx}}(\mathcal{M})&:=\inf_{\bm{X}\in\mathcal{M}}\sup \left\{r>0:  \text {each ball is strongly convex } \text {and each geodesic is minimal in } \mathbf{B}(\bm{X}, r)\right\}.
\end{align*}
}  is
\[
r_{{\rm inj}}(\og)=\sqrt{2}\pi\quad \text{and}\quad r_{{\rm cvx}}(\og)=\frac{\sqrt{2}}{2}\pi.
\]
\end{lemma}

\begin{remark}\label{convergence-key-remark}
We discuss the relationship between $\|\bm{E}^\star\|_{\infty}$ (resp. $\|\bm{E}^\star\|_{F}$) and the Euclidean distance $\|\bG\bQ^\star-\bG^\star\|_\infty$ (resp. $\operatorname{d}_F([\bG], [\bG^\star])$). 
Actually from \cite[Proposition 10.22]{boumal2023introduction} we know that it is the difference between Riemannian and Euclidean distance since $\operatorname{d}^{\og}(\bG\bQ^\star, \bG^\star)=\|\bm{E}^\star\|_{F}$ if $\|\bm{E}^\star\|_{\infty}<r_{\rm inj}(\og)=\sqrt{2}\pi$ which is always satisfied under our assumption in this paper. From \eqref{Euc-Rie-dist-infty} and \eqref{Euc-Rie-dist-F} we know that the Euclidean distance is controlled by the Riemannian one. On the other hand, for example when $d=3$ one has for each $i\in[n]$ that
\begin{align*}
\|\bG_i \bQ^\star - \bG_i^\star\|_F
&=\|\exp(\bm{E}_i^\star)-\bm{I}_3\|_F\notag\\
&=\left\|\frac{\sqrt{2}\sin\left(\frac{\sqrt{2}}{2}\|\bm{E}_i^\star\|_F\right)}{\|\bm{E}_i^\star\|_F}\bm{E}_i^\star+\frac{2\left(1-\cos\left(\frac{\sqrt{2}}{2}\|\bm{E}_i^\star\|_F\right)\right)}{\|\bm{E}_i^\star\|_F^2}\bm{E}_i^{\star2}\right\|_F\notag\\
&=\sqrt{2\sin \left(\frac{\sqrt{2}}{2}\|\bm{E}_i^\star\|_F\right)^2+4\left(\cos \left(\frac{\sqrt{2}}{2}\|\bm{E}_i^\star\|_F\right)-1\right)^2\cdot\frac{\|\bm{E}_i^{\star2}\|_F^2}{\|\bm{E}_i^\star\|_F^4}}\notag\\
&\geq\sqrt{2}\sin \left(\frac{\sqrt{2}}{2}\|\bm{E}_i^\star\|_F\right)\notag,
\end{align*}
and consequently  we know that 
$\|\bG\bQ^\star-\bG^\star\|_{\infty}\leq \frac{1}{10}$ implies that $\|\bG\bQ^\star-\bG^\star\|_{\infty}\ge \frac{99}{100}\|\bm{E}^\star\|_{\infty}$. 
Then for simplicity we can use $\|\bG\bQ^\star-\bG^\star\|_\infty$ (resp. $\operatorname{d}_F([\bG], [\bG^\star])$) to replace the assumptions in Proposition \ref{prop: stayintheball} about $\|\bm{E}^\star\|_{\infty}$ (resp. $\|\bm{E}^\star\|_{F}$).
\end{remark}
\begin{remark}\label{transfer-point-remark}
All assumptions in the main results (e.g., Theorem \ref{hess-positive}, Corollary \ref{cor-EB}, Theorem \ref{thm:convergence-rate}) about $\operatorname{d}_F([\bG],[\widehat{\bG}])$ (resp. $\|\bG\widehat{\bQ}-\widehat{\bG}\|_\infty$) can be transferred to $\operatorname{d}_F([\bG], [\bG^\star])$ (resp. $\|\bG\bQ^\star-\bG^\star\|_\infty$) with the help of Proposition \ref{prop-sol-in}. Actually, the triangle inequality and Lemma \ref{dist-GstarGhat} imply that
\begin{align*}
  \operatorname{d}_F([\bG],[\widehat{\bG}]) \leq \operatorname{d}_F([\bG],[\bG^\star]) + \operatorname{d}_F([\widehat{\bG}], [\bG^\star]) \leq \operatorname{d}_F([\bG], [\bG^\star]) + \frac{4\sqrt{d}\|\bm{\Delta}\|}{\sqrt{n}}.
\end{align*}
This together with Lemma \ref{lemma-dist-rotation} indicates that (without loss of generality $\bQ^\star=\bm{I}_d$)
\begin{align*}
\|\bG\widehat{\bQ}-\widehat{\bG}\|_\infty
&=\|\bG\widehat{\bQ}-\bG\widehat{\bQ}^\star+\bG\widehat{\bQ}^\star-\bG^\star \bQ^\star\widehat{\bQ}^\star+\bG^\star \bQ^\star\widehat{\bQ}^\star-\bG^\star\widehat{\bQ}^\star+\bG^\star\widehat{\bQ}^\star-\widehat{\bG}\|_\infty\notag\\
&=\|\widehat{\bQ}-\widehat{\bQ}^\star\|_F+\|\bG\bQ^\star-\bG^\star\|_\infty+\|\bQ^\star-\bm{I}_d\|_F+\|\bG^\star\widehat{\bQ}^\star-\widehat{\bG}\|_\infty \notag\\
&\leq \|\bQ^{\star\top}\widehat{\bQ}\widehat{\bQ}^{\star\top}-\bm{I}_d\|_F+2\|\bQ^\star-\bm{I}_d\|_F+\|\bG\bQ^\star-\bG^\star\|_\infty+\|\bG^\star\widehat{\bQ}^\star-\widehat{\bG}\|_\infty\notag\\
&\leq 4 \varepsilon\cdot \frac{\operatorname{d}_{F}([\bG],[\widehat{\bG}])}{\sqrt{n}}+\|\bG\bQ^\star-\bG^\star\|_\infty+\|\bG^\star\widehat{\bQ}^\star-\widehat{\bG}\|_\infty.
\end{align*}
This make the transfer with same order under our noise level and effective radius setting, e.g., when $\|\bm{\Delta}\|\le \frac{n^{3/4}}{20d^{1/2}}$, $\|\bm{\Delta} \bG^{\star}\|_{\infty}\le \frac{n}{20}$ and
$\rho_F=\frac{1}{10}\min\{\sqrt{n},\frac{n}{\|\bm{\Delta}\|}\}$,
$\rho_{\infty}=\frac{1}{4}$ we have that if
\[
\operatorname{d}_F([\bG], [\bG^\star])=\mathcal{O}(\rho_F),\quad \|\bG\bQ^\star-\bG^\star\|_\infty=\mathcal{O}(\rho_{\infty})
\]
then it follows that
\[
\operatorname{d}_F([\bG],[\widehat{\bG}])
\leq\operatorname{d}_F([\bG], [\bG^\star]) + \min\left\{ n^{1/4}, \frac{n}{\|\bm{\Delta}\|}\right\}=\mathcal{O}(\rho_F),
\]
and
\[
\|\bG\widehat{\bQ}-\widehat{\bG}\|_\infty
\leq \|\bG\bQ^\star-\bG^\star\|_\infty+\mathcal{O}(1)=\mathcal{O}(\rho_{\infty}).
\]
\end{remark}

Now, we focus on the initialization which is also crucial known from Theorem \ref{thm:convergence-rate}. The spectral estimator can provide a initialization satisfies assumptions required in Theorem \ref{thm:convergence-rate} for (quotient) Riemannian gradient method. The following lemma illustrate this result by quantifying the distance between an initial point $\bG^0$ generated by the spectral estimator and the ground truth $\bG^\star$.

\begin{lemma}[Spectral Estimation Error]\label{lemma: spectral}
Suppose that $\|\bm{\Delta}\|\le \frac{n^{3/4}}{20d^{1/2}}$, $\|\bm{\Delta} \bG^{\star}\|_{\infty}\le \frac{n}{40d}$.
Then for the ground truth $\bG^\star\in\ogn$, the spectral estimator $\bG^0=\Pi_{\ogn}(\bm{\Phi}) \in\ogn$ satisfies that 
  \begin{align} 
    \operatorname{d}_F([\bG^0], [\bG^\star]) \leq \frac{8\sqrt{d}  \|\bm{\Delta} \|  }{\sqrt{n}}\quad\text{and}\quad \| \bG^0\bQ_0^\star - \bG^\star \|_\infty\leq\frac{16\left\|\bm{\Delta} \bG^\star\right\|_{\infty}}{n}+\frac{8\sqrt{d}  \|  \bm{\Delta} \|  }{n},
  \end{align}
where $\bQ_0^\star\in\od$ satisfies $\operatorname{d}_F([\bG^0],[\bG^\star])=\Vert \bG^0 \bQ_0^\star-\bG^\star\Vert_F$.
\end{lemma}
\begin{proof}
Recall that $\bm{\Phi}$ is the matrix of top $d$ eigenvectors of $\bC$ with $\bm{\Phi}^\top \bm{\Phi} = n\bm{I}_d$, which satisfies $\operatorname{tr}(\bm{\Phi}^\top \bC \bm{\Phi}) \geq \operatorname{tr} (\bG^{\star\top} \bC \bG^{\star} )$. Without loss of generality, we assume $\operatorname{d}_F([\bm{\Phi}], [\bG^\star]) = \| \bm{\Phi} - \bG^\star \|_F$. Then it follows from \eqref{dist-GstarGhat-ineq} that
  \begin{align*}
    \operatorname{d}_F([\bG^0], [\bG^\star]) \leq \| \bG^0 - \bG^\star \|_F \leq 2 \| \bm{\Phi} - \bG^\star\|_F  \leq \frac{8\sqrt{d}  \|  \bm{\Delta} \|  }{\sqrt{n}},
  \end{align*}
where the second inequality is due to \cite[Lemma 2]{liu2020unified} for $\ogn$. 
Next, since from Von Neumann's inequality $d\cdot\sigma_{\min}(\bC_{i,:}\bm{\Phi})\leq \sum_{l=1}^d \sigma_{l}(\bC_{i,:}\bm{\Phi}) 
\leq \sigma_{\max}(\bC)\cdot\|\bm{\Phi}_i\|_*$ for each $i\in[n]$, we know from \eqref{singular-lowbd} that
\[
\|\bm{\Phi}_i\|_*\ge\frac{d\cdot\sigma_{\min}(\bC_{i,:}\bm{\Phi})}{\sigma_{\max}(\bC)}\ge \frac{d\cdot\left(n-\frac{8d\|\bm{\Delta}\|^2}{n}-\|\bm{\Delta}\|\| \bm{\Phi}-\bG^\star\|_{F}-\|\bm{\Delta}\bG^\star\|_{\infty}\right)}{n+\|\bm{\Delta}\|}\ge d-\frac{1}{32}.
\]
Then $\operatorname{d}_F([\bG^0], [\bm{\Phi}])^2=2nd-2\operatorname{tr}(\bG^{0\top}\bm{\Phi})
=2nd-2\sum_{i=1}^n\|\bm{\Phi}_i\|_*\le \frac{n}{16}$ and this together with Lemma \ref{lemma-dist-rotation} (also $\operatorname{d}_F([\bm{\Phi}], [\bG^\star])\leq \frac{4\sqrt{d}  \|  \bm{\Delta} \|  }{\sqrt{n}}$) implies that with $\varepsilon=\frac{1}{4}$,
\begin{equation}\label{spectral-key}
\| \bQ_0^\star - \bm{I}_d \|_F\leq 4 \varepsilon\cdot \frac{\operatorname{d}_{F}([\bG^0],[\bG^\star])}{\sqrt{n}}.
\end{equation}
Thus, it follows \eqref{distinf-GstarGhat-ineq} and \eqref{spectral-key} that
\begin{align*}
\| \bG^0\bQ_0^\star - \bG^\star \|_\infty
&\leq \| \bG^0 - \bG^\star \|_\infty+\| \bQ_0^\star - \bm{I}_d \|_F \notag\\
&\leq 2\|\bm{\Phi}-\bG^\star\|_{\infty}+\| \bQ_0^\star - \bm{I}_d \|_F\notag\\
&\leq\frac{16\left\|\bm{\Delta} \bG^\star\right\|_{\infty}}{n}+\frac{8\sqrt{d}  \|  \bm{\Delta} \|  }{n}.
\end{align*}
The proof is complete.
\end{proof}

Combining the results in Theorem \ref{thm:convergence-rate}, Proposition \ref{prop: stayintheball}, Lemma \ref{lemma: spectral} and Remark \ref{convergence-key-remark} and \ref{transfer-point-remark} we have the following corollary about the convergence result of (quotient) Riemannnian gradient method with spectral initialization.

\begin{cor}
Suppose that $\|\bm{\Delta}\|\le \frac{n^{3/4}}{80d^{1/2}}$, $\|\bm{\Delta} \bG^{\star}\|_{\infty}\le \frac{n}{400d}$,
  and the sequence $\{ \bG^k\}_{k\geq 0}$ generated by Algorithm \ref{alg: rgd} with spectral initialization and stepsize $0<\underline{t}\leq t_k\leq \frac{1}{4nd}$ for any $k\geq 0$.  
Then the sequence $\{f([\bG^k])\}_{k\ge0}$ (resp. $\{\bG^k\}_{k\geq 0}$) converges Q-linearly (resp. R-linearly) to $f([\widehat{\bG}])$ (resp. some $\bG^*\in[\widehat{\bG}]$) globally.
\end{cor}

\begin{proof}
It suffices to verify the assumptions in Theorem \ref{thm:convergence-rate} to derive the desired results. From Lemma \ref{lemma: spectral}, Proposition \ref{prop: stayintheball} and Remark \ref{convergence-key-remark}, we know that the iterates satisfy \eqref{stay-in-ball-requ}  for all $k\ge0$ by transferring the distance measurement from the point $\bG^\star$ to $\widehat{\bG}$ by Remark  \ref{transfer-point-remark} (i.e., starting and staying in the effective region of (Riemannian) local error bound given in Corollary \ref{cor-EB}). Also, it follows for each $k\ge0$ that
\begin{equation*}
\begin{aligned}
\|\bm{\Delta} \bG^k\|_{\infty}
\leq \| \bm{\Delta} (\bG^k \bQ^{k}-\bG^{\star})\|_{\infty}+\left\|\bm{\Delta} \bG^{\star}\right\|_{\infty} 
&\leq \| \bm{\Delta}\|\cdot \operatorname{d}_F([\bG^k], [\bG^{\star}])+\left\|\bm{\Delta} \bG^{\star}\right\|_{\infty}\\
&\leq \frac{1}{200}\cdot\min\left\{\sqrt{n}\| \bm{\Delta}\|,n\right\}+\frac{n}{400d}\leq n,
\end{aligned}
\end{equation*}
where $\bQ^k\in \mathcal{O}(d)$ satisfies $\operatorname{d}_F([\bG^k],[\widehat{\bG}])=\Vert \bG^k \bQ^k-\widehat{\bG} \Vert_F$.
Then it follows that 
\[
n(d+1)+\|\bm{\Delta}\|+\sqrt{d}\cdot\|\bm{\Delta} \bG^k\|_{\infty}\leq 3nd
\]
and consequently with $\alpha\in(0,\frac{1}{4}]$ we know for each $k\ge0$ that
\begin{equation*}
  t_k \leq\frac{1}{4nd}\leq \frac{1-\alpha}{3nd}\leq \frac{1-\alpha}{n(d+1)+\|\bm{\Delta}\|+\sqrt{d}\cdot\|\bm{\Delta} \bG^k\|_{\infty}},
  \end{equation*}
which illustrates that the stepsize $\{t_k\}_{k\geq 0}$ satisfying \eqref{step-req}. Thus, all assumptions in Theorem \ref{thm:convergence-rate} are satisfied and the proof is complete.
\end{proof}

\section{Conclusion}

In this work, we study the landscape of least-squares formulation of the orthogonal group synchronization problem from the quotient geometric view. Local strongly concave property on the quotient manifold is proved under certain noise level, and as a byproduct we derive the (Riemannian) local error bound property on both the original and quotient manifold. Improved estimation results of the least-squares and spectral estimator are derived with near-optimal noise level for exact recovery. As an algorithmic consequence, the sequential linear convergence
result of (quotient) Riemannian gradient method (with spectral initialization) to the global maximizers is proved. For future directions, it would be interesting to study the convergence properties of second-order method which is quite different from the Riemannian gradient method whose iterative direction is deviated on the original and quotient manifold.


\label{section-con}

\section*{\LARGE Appendix}
\begin{appendices}
\section{Proof of Lemma \ref{lemma-exp-lip}
}
\label{appendix-A}

From basic calculation we know that
\begin{align*}
\frac{\operatorname{d}}{\mathrm{d}t} \exp (\bm{E}+t \bm{\Delta} \bm{E})\bigg|_{t=0} 
&=\lim _{n \rightarrow \infty} \frac{\operatorname{d}}{\mathrm{d} t}\left(\bm{I}_d+\frac{\bm{E}+t \bm{\Delta} \bm{E}}{n}\right)^{n}\bigg|_{t=0}\\ 
&=\lim _{n \rightarrow \infty} \sum_{k=0}^{n}\left(\bm{I}_d+\frac{\bm{E}}{n}\right)^{k} \frac{\bm{\Delta}\bm{E}}{n} \left(\bm{I}_d+\frac{\bm{E}}{n}\right)^{n-k-1}\\
&=\int_{0}^{1} \exp (s \bm{E})\cdot \bm{\Delta} \bm{E}\cdot \exp ((1-s) \bm{E}) \mathrm{d} s.
\end{align*}
Then by the mean value inequality it follows that
\begin{align*}
\|\exp(\bm{E}')-\exp(\bm{E})\|_F
&\leq\sup_{c\in[0,1]}\|\operatorname{D}\exp(\bm{E}+c(\bm{E}'-\bm{E}))\|\|\bm{E}'-\bm{E}\|_F
\leq \|\bm{E}'-\bm{E}\|_F,\notag
\end{align*}
where the last inequality is from the fact that for any $\bm{E}\in\operatorname{Skew}(d)$
\begin{align*}
\|\operatorname{D}\exp(\bm{E})\|=\sup_{\|\bm{\Delta} \bm{E}\|_F=1}\|\operatorname{D}\exp(\bm{E})\bm{\Delta} \bm{E}\|_F
&=\sup_{\|\bm{\Delta} \bm{E}\|_F=1}\left\|\frac{\operatorname{d}}{\mathrm{d} t} \exp (\bm{E}+t \bm{\Delta} \bm{E})\bigg|_{t=0}\right\|_F\notag\\
&=\sup_{\|\bm{\Delta} \bm{E}\|_F=1}\left\|\int_{0}^{1} \exp (s \bm{E})\cdot \bm{\Delta} \bm{E}\cdot \exp ((1-s) \bm{E})\mathrm{d} s\right\|_F\notag\\
&\leq \sup_{\|\bm{\Delta} \bm{E}\|_F=1}\|\bm{\Delta} \bm{E}\|_F\cdot \int_{0}^{1} 1\mathrm{d} s=1.
\end{align*}
Thus, the inequality \eqref{retr-key-ineq-1} is directly from 
\begin{align}
\left\|\operatorname{Exp}_{\bG}(\bm{\xi})-\bG\right\|_{F}
&=\left\|[\bG_1\exp(\bm{E}_1)-\bG_1;\ldots; \bG_n\exp(\bm{E}_n)-\bG_n]\right\|_F \notag\\
&=\sqrt{\sum_{i=1}^n\|\exp(\bm{E}_i)-\bm{I}_d\|_F^2}\notag\\
&\leq \sqrt{\sum_{i=1}^n\|\bm{E}_i\|_F^2}=\|\bm{\xi}\|_F.\notag
\end{align}
For \eqref{retr-key-ineq-2}, we know that
\begin{align}\label{lip-key-0}
&\|\operatorname{D} \exp(\bm{E}')-\operatorname{D} \exp(\bm{E})\|\notag\\
= &\ \sup_{\|\bm{\Delta} \bm{E}\|_F=1}\|\operatorname{D} \exp(\bm{E}')\bm{\Delta} \bm{E}-\operatorname{D} \exp(\bm{E})\bm{\Delta} \bm{E}\|_F\notag\\
= &\ \sup_{\|\bm{\Delta} \bm{E}\|_F=1}\left\|\left(\int_{0}^{1} \exp (s \bm{E}')\cdot \bm{\Delta} \bm{E}\cdot \exp ((1-s) \bm{E}')\mathrm{d} s\right)-\left(\int_{0}^{1} \exp (s \bm{E})\cdot \bm{\Delta} \bm{E}\cdot \exp ((1-s) \bm{E})\mathrm{d} s\right) \right\|_F\notag\\
= &\ \sup_{\|\bm{\Delta} \bm{E}\|_F=1}\bigg\{\left\|\left(\int_{0}^{1} \exp (s \bm{E}')\cdot \bm{\Delta} \bm{E}\cdot \exp ((1-s) \bm{E}')\mathrm{d} s\right)-\left(\int_{0}^{1} \exp (s \bm{E}')\cdot \bm{\Delta} \bm{E}\cdot \exp ((1-s) \bm{E})\mathrm{d} s\right) \right\|_F\notag\\
&+\left\|\left(\int_{0}^{1} \exp (s \bm{E}')\cdot \bm{\Delta} \bm{E}\cdot \exp ((1-s) \bm{E})\mathrm{d} s\right)-\left(\int_{0}^{1} \exp (s \bm{E})\cdot \bm{\Delta} \bm{E}\cdot \exp ((1-s) \bm{E})\mathrm{d} s\right) \right\|_F\bigg\}\notag\\
\leq &\ \sup_{\|\bm{\Delta} \bm{E}\|_F=1}\left\{\|\bm{\Delta} \bm{E}\|_F \int_0^1  \|\exp ((1-s) \bm{E}')-\exp ((1-s) \bm{E})\|_F+\|\exp (s \bm{E}')-\exp (s \bm{E})\|_F\mathrm{d}s \right\}\notag\\
\leq &\ \|\bm{E}'-\bm{E}\|_F \int_0^1 1 \mathrm{d}s=\|\bm{E}'-\bm{E}\|_F.
\end{align}
Note further that
\[
\exp(\bm{E}')-\exp(\bm{E})-\operatorname{D}\exp(\bm{E})(\bm{E}'-\bm{E})
 =\int_0^1\operatorname{D}  \exp(\bm{E}+s(\bm{E}'-\bm{E}))(\bm{E}'-\bm{E})\mathrm{d}s-\operatorname{D}  \exp(\bm{E})(\bm{E}'-\bm{E}).
 \]
This, together with \eqref{lip-key-0} implies that 
\begin{align*}
&\|\exp(\bm{E}')-\exp(\bm{E})-\operatorname{D}\exp(\bm{E})(\bm{E}'-\bm{E})\|_F\notag\\
\le &\ \int_0^1\|(\operatorname{D}  \exp(\bm{E}+s(\bm{E}'-\bm{E}))-\operatorname{D}  \exp(\bm{E}))(\bm{E}'-\bm{E})\|_F \mathrm{d}s\notag\\
\le &\ \int_0^1\|\operatorname{D}  \exp(\bm{E}+s(\bm{E}'-\bm{E}))-\operatorname{D}  \exp(\bm{E})\| \mathrm{d}s\cdot \|\bm{E}'-\bm{E}\|_F\notag\\
\le &\  \int_0^1 s\|\bm{E}'-\bm{E}\|_F\mathrm{d}s\cdot\|\bm{E}'-\bm{E}\|_F=\frac{1}{2}\|\bm{E}'-\bm{E}\|_F^2.
\end{align*}
Consequently, one has that
\begin{align}
\left\|\operatorname{Exp}_{\bG}(\bm{\xi})-(\bG+\bm{\xi})\right\|_{F}
&=\left\|[\bG_1\exp(\bm{E}_1)-\bG_1-\bG_1\bm{E}_1;\ldots; \bG_n\exp(\bm{E}_n)-\bG_n-\bG_n\bm{E}_n]\right\|_F \notag\\
&=\sqrt{\sum_{i=1}^n\|\exp(\bm{E}_i)-\bm{I}_d-\bm{E}_i\|_F^2}\notag\\
&\leq \sqrt{\sum_{i=1}^n\frac{1}{4}\|\bm{E}_i\|_F^4}\leq\frac{1}{2}\|\bm{\xi}\|_F^2.\notag
\end{align}
The proof is complete.

\section{Proof of Lemma \ref{lemma-dist-rotation}}
\label{appendix-B}
Without loss of generality, we assume that $\operatorname{d}_{F}([\bm{H}_1], [\bG]) \leq \operatorname{d}_{F}([\bm{H}_2], [\bG])$. We decompose $\bm{H}_1, \bm{H}_2$ into two parts 
\[\bm{H}_1=a \bG+\sqrt{n} \bm{W}\quad \text{and} \quad \bm{H}_2=b \bG+\sqrt{n} \bm{V},\] 
where $\langle \bG, \bm{W}\rangle=\langle \bG, \bm{V}\rangle=0$ and $a, b$ are nonnegative real numbers (which is guaranteed by the given assumptions). 
By the definition of $\bQ$, we have $\left\|\bm{H}_1 \bQ-\bm{H}_2\right\|_{F}^{2} \leq\|\bm{H}_1-\bm{H}_2\|_{F}^{2}$, and from the decomposition of $\bm{H}_1, \bm{H}_2$ we deduce that $n\|a\bQ-b\bm{I}_d\|_F^2+n\|\bm{W}\bQ-\bm{V}\|_F^2\leq n\|(a-b)\bm{I}_d\|_F^2+n\|\bm{W}-\bm{V}\|_F^2$, i.e.,
\begin{align}\label{GHlemma-key1}
\|a\bQ-b\bm{I}_d\|_F^2-\|(a-b)\bm{I}_d\|_F^2
&\leq \|\bm{W}-\bm{V}\|_F^2-\|\bm{W}\bQ-\bm{V}\|_F^2 \notag\\
&=\left(\|\bm{W}-\bm{V}\|_{F}-\left\|\bm{W}\bQ-\bm{V}\right\|_{F}\right)\left(\|\bm{W}-\bm{V}\|_{F}+\left\|\bm{W}\bQ-\bm{V}\right\|_{F}\right) \notag\\
& \leq\|\bm{I}_d-\bQ\|_F\|\bm{W}\|_{F}\left(\|\bm{W}-\bm{V}\|_{F}+\left\|\bm{W}\bQ-\bm{V}\right\|_{F}\right) \notag\\
& \leq\|\bm{I}_d-\bQ\|_F\|\bm{W}\|_{F}\left(\|\bm{I}_d-\bQ\|_F\|\bm{W}\|_{F}+2\left\|\bm{W}\bQ-\bm{V}\right\|_{F}\right),
\end{align}
where the last two inequalities are from the triangle inequality. Since $\|a\bQ-b\bm{I}_d\|_F^2-\|(a-b)\bm{I}_d\|_F^2=\|\bm{I}_d-\bQ\|_F^2 ab$, then \eqref{GHlemma-key1} reduces to
\begin{equation}\label{GHlemma-key2}
\|\bm{I}_d-\bQ\|_F ab \leq\|\bm{W}\|_{F}\left(\|\bm{I}_d-\bQ\|_F\|\bm{W}\|_{F}+2\left\|\bm{W}\bQ-\bm{V}\right\|_{F}\right).
\end{equation}
Since $\operatorname{d}_{F}([\bm{H}_1], [\bG])^{2}=\|\bm{H}_1-\bG\|_{F}^{2}=n\|(1-a)\bm{I}_d\|_F^{2}+n\|\bm{W}\|_{F}^{2}=nd|1-a|+n\|\bm{W}\|_{F}^{2}$, we have $\|\bm{W}\|_{F} \leq \frac{\operatorname{d}_{F}([\bm{H}_1], [\bG])}{\sqrt{n}}$. This combined with \eqref{GHlemma-key2} yields
\begin{equation}\label{rotation-angle-key-1}
\|\bQ-\bm{I}_d\|_F \left(\frac{a b\sqrt{n}}{\operatorname{d}_{F}([\bm{H}_1], [\bG])}-\frac{\operatorname{d}_{F}([\bm{H}_1], [\bG])}{\sqrt{n}}\right) \leq 2\left\|\bm{W}\bQ-\bm{V}\right\|_{F}.
\end{equation}
From $\operatorname{d}_{F}([\bm{H}_1], [\bG])^{2}=2(nd-\|\bm{H}_1^\top \bG\|_*)=2(nd-na\|\bm{I}_d\|_*)=2 nd(1-a)$ and $\operatorname{d}_{F}([\bm{H}_2], [\bG])^{2}=2 nd(1-b)$, we know that
\begin{equation*}
a b =\left(1-\frac{\operatorname{d}_{F}([\bm{H}_1], [\bG])^{2}}{2nd}\right)\left(1-\frac{\operatorname{d}_{F}([\bm{H}_2], [\bG])^2}{2nd}\right) \geq 1-\frac{\operatorname{d}_{F}([\bm{H}_1], [\bG])^{2}+\operatorname{d}_{F}([\bm{H}_2], [\bG])^{2}}{2nd}.
\end{equation*}
Then we have
\begin{align}\label{rotation-angle-key-2}
&\frac{a b\sqrt{n}}{\operatorname{d}_{F}([\bm{H}_1], [\bG])}-\frac{\operatorname{d}_{F}([\bm{H}_1], [\bG])}{\sqrt{n}}\notag\\
\ge&\ \frac{\sqrt{n}}{\operatorname{d}_{F}([\bm{H}_1], [\bG])}-\frac{\operatorname{d}_{F}([\bm{H}_1], [\bG])^2+\operatorname{d}_{F}([\bm{H}_2], [\bG])^2}{2\sqrt{n}d\cdot\operatorname{d}_{F}([\bm{H}_1], [\bG])}-\frac{\operatorname{d}_{F}([\bm{H}_1], [\bG])}{\sqrt{n}}\notag\\
=&\ \frac{\sqrt{n}}{\operatorname{d}_{F}([\bm{H}_1], [\bG])}\left(1-\frac{\operatorname{d}_{F}([\bm{H}_1], [\bG])^{2}+\operatorname{d}_{F}([\bm{H}_2], [\bG])^{2}}{2nd}-\frac{\operatorname{d}_{F}([\bm{H}_1], [\bG])^2}{n}\right)\notag\\
=&\ \frac{\sqrt{n}}{\operatorname{d}_{F}([\bm{H}_1], [\bG])}\left(1-\frac{(2d+1)\operatorname{d}_{F}([\bm{H}_1], [\bG])^{2}+\operatorname{d}_{F}([\bm{H}_2], [\bG])^{2}}{2nd}\right)\notag\\
\ge&\ \frac{\sqrt{n}}{\operatorname{d}_{F}([\bm{H}_1], [\bG])}\left(1-\frac{2\operatorname{d}_{F}([\bm{H}_2], [\bG])^{2}}{n}\right)
\end{align}
where the last inequality is from $\operatorname{d}_{F}([\bm{H}_1], [\bG]) \leq \operatorname{d}_{F}([\bm{H}_2], [\bG])$ and the fact that $d+1\leq 2d$. Thus, from \eqref{rotation-angle-key-1} and \eqref{rotation-angle-key-2} we obtain
\begin{align*}
\|\bm{I}_d-\bQ\|_F  
&\leq \frac{ 2\sqrt{n}\min\{\operatorname{d}_{F}([\bm{H}_1], [\bG]),\operatorname{d}_{F}([\bm{H}_2], [\bG])\}}{n-2\max\{\operatorname{d}_{F}([\bm{H}_1], [\bG])^2,\operatorname{d}_{F}([\bm{H}_2], [\bG])^2\}} \cdot \left\|\bm{W}\bQ-\bm{V}\right\|_{F} \notag\\
&\leq \frac{ 2\min\{\operatorname{d}_{F}([\bm{H}_1], [\bG]),\operatorname{d}_{F}([\bm{H}_2], [\bG])\}}{n-2\max\{\operatorname{d}_{F}([\bm{H}_1], [\bG])^2,\operatorname{d}_{F}([\bm{H}_2], [\bG])^2\}} \cdot \operatorname{d}_{F}([\bm{H}_1], [\bm{H}_2])\notag\\
&\leq \frac{4\min\{\varepsilon_1,\varepsilon_2\}}{\sqrt{n}}\cdot \operatorname{d}_{F}([\bm{H}_1], [\bm{H}_2]).\notag
\end{align*}
The proof is complete.

\section{Proof of Lemma \ref{lemma-inj-cvx-radius}}
\label{appendix-C}
We present the proof on $\sod$ for example. From \cite[Corollary 5.7]{cheeger2008comparison} we know that 
\begin{equation}\label{inj-sod}
r_{\operatorname{inj}}(\sod)=\min \left\{\operatorname{conj}(\sod), \frac{l}{2}\right\}.
\end{equation}
Here, $l$ is the length of the shortest periodic (or closed) geodesic and $\operatorname{conj}(\sod)$ is the conjugate radius of $\sod$, which satisfies
\[
\operatorname{conj}(\sod) \geq \frac{\pi}{\sqrt{\bar{\kappa}}},
\]
where $\bar{\kappa}$ is an upper bound on the sectional curvature of $\sod$. For  orthonormal matrices $\bm{X}$, $\bm{Y}$, the sectional curvature (cf.\ \cite[Corollary 3.19]{cheeger2008comparison}) satisfies 
\[
K(\bm{X}, \bm{Y})=\frac{1}{4}\|[\bm{X}, \bm{Y}]\|_{F}^{2}\leq \frac{1}{2}\|\bm{X}\|_F^2\|\bm{Y}\|_F^2\leq \frac{1}{2},
\]
where $[\cdot, \cdot]$ is the Lie bracket and the first inequality is from \cite[Theorem 1]{audenaert2010variance}. 

On the other hand, for the shortest periodic geodesic in $\sod$ (without loss of generality it starts and ends at $\bm{I}_d$), we need to find $\bm{E} \in \operatorname{Skew}(d)$ and $\bm{E}\neq\bm{0}$ such that $\exp(\bm{E})=\bm{I}_d$. Consider the Schur decomposition of the skew-symmetric matrix $\bm{E}=\bm{U} \bm{\Lambda} \bm{U}^{\top}$, where
\begin{equation}
\bm{\Lambda}=\left[\begin{array}{ccccc}
0 & -\theta_1 & & & \\
\theta_1 & 0 & & & \\
& &\ddots  &  & \\
& &  & 0 & -\theta_{\lfloor d/2\rfloor}\\
& &  & \theta_{\lfloor d/2\rfloor} & 0
\end{array}\right].
\end{equation}
Then the constraint $\exp(\bm{\Lambda})=\bm{I}_d$ implies that $\theta_{i}=2 \pi k_{i} $, $k_{i} \in \mathbb{Z}$, which implies the shortest length of non-trivial minimizing geodesic in $\sod$ is $2\sqrt{2} \pi$.
Hence, from \eqref{inj-sod} we know that $r_{\operatorname{inj}}(\sod)=\sqrt{2}\pi$, 
and consequently it follows from \cite[Definition 2.3]{afsari2013convergence} that
\[
r_{\operatorname{cvx}}(\sod)=\frac{1}{2} \min \left\{r_{\operatorname{inj}}(\sod), \frac{\pi}{\sqrt{\bar{\kappa}}}\right\}=\frac{\sqrt{2}}{2}\pi.
\]
The proof is complete.

\end{appendices}

\nocite{}
\bibliography{references}
\bibliographystyle{alpha}
\end{document}